\documentclass[11pt,a4paper]{article}

\usepackage{graphicx}        

\usepackage[a4paper, hmargin={2.7cm,2.7cm},vmargin={3.6cm,2.9cm}]{geometry}
\usepackage{amsmath, amssymb, amsfonts, amsbsy, latexsym, xcolor,dsfont}
\usepackage{amscd,amsxtra}
\usepackage{amsthm}

\usepackage{enumitem}

\usepackage[T1]{fontenc}
\usepackage[utf8]{inputenc}
\usepackage[english]{babel}
\usepackage{twoopt}
\usepackage{doi}

\usepackage[bf,compact,pagestyles,small]{titlesec}
\titlespacing*{\section}{0pt}{14pt}{4pt}
\titlespacing*{\subsection}{0pt}{8pt}{3pt}

\usepackage{etoolbox}
\makeatletter
\patchcmd{\ttlh@hang}{\parindent\z@}{\parindent\z@\leavevmode}{}{}
\patchcmd{\ttlh@hang}{\noindent}{}{}{}
\makeatother

\usepackage{mathdots}

\usepackage{hyperref}
\hypersetup{
    colorlinks,
    linkcolor={red!50!black},
    citecolor={blue!50!black},
    urlcolor={blue!80!black},
 pdfview={FitH},
 pdfstartview={FitH},
 bookmarksnumbered={true},
 pdfauthor = {Felix Voigtlaender},
 pdftitle = {The universal approximation theorem for complex-valued neural networks},
 pdfsubject = {},
 pdfkeywords = {},
 pdfcreator = {LaTeX with hyperref package},
 pdfproducer = PDFlatex}

\usepackage[noabbrev]{cleveref}

\usepackage{silence}
\usepackage[noadjust]{cite}
\usepackage{todonotes}
\usepackage{url}

\usepackage{fancyhdr,lastpage}
\def\maketimestamp{\count255=\time
\divide\count255 by 60\relax
\edef\thetime{\the\count255:}%
\multiply\count255 by-60\relax
\advance\count255 by\time
\edef\thetime{\thetime\ifnum\count255<10 0\fi\the\count255}
\edef\thedate{\number\day-\ifcase\month\or Jan\or Feb\or Mar\or
             Apr\or May\or Jun\or Jul\or Aug\or Sep\or Oct\or
             Nov\or Dec\fi-\number\year}
\def\timstamp{\hbox to\hsize{\tt\hfil\thedate\hfil\thetime\hfil}}}
\maketimestamp
\fancypagestyle{plain}{%
\fancyhf{} 
\fancyfoot[C]{\small \thepage{} of \pageref{LastPage}}}
\fancypagestyle{firststyle}
{
   \fancyhf{}
   \fancyfoot[L]{\scriptsize \tt date/time:\,\thedate/\thetime}
   \fancyfoot[C]{\small  \thepage{} of \pageref{LastPage}}
   \fancyfoot[R]{\scriptsize \tt file:\jobname.tex}
}

\theoremstyle{plain}
\newtheorem{theorem}{Theorem}[section]
\newtheorem{proposition}[theorem]{Proposition}
\newtheorem*{proposition*}{Proposition}
\newtheorem{lemma}[theorem]{Lemma}

\theoremstyle{definition}
\newtheorem{definition}[theorem]{Definition}

\theoremstyle{remark}
\newtheorem{remark}[theorem]{Remark}
\newtheorem{example}[theorem]{Example}
\theoremstyle{remark}
\newtheorem*{rem*}{Remark}

\numberwithin{equation}{section}

\newcommand{\CalN}{\mathcal{N}}

\newcommand{\CalW}{\mathcal{W}}

\newcommand{\CalF}{\mathcal{F}}

\newcommand{\CalH}{\mathcal{H}}

\newcommand{\CalG}{\mathcal{G}}

\newcommand{\CalM}{\mathcal{M}}

\newcommand{\linspan}{\operatorname{span}}

\newcommand{\NN}{\mathcal{NN}}
\newcommand{\id}{\mathrm{id}}

\DeclareMathOperator\arctanh{arctanh}
\DeclareMathOperator\arcsinh{arcsinh}

\def\supp{\mathop{\operatorname{supp}}}

\newcommand{\R}{\mathbb{R}}

\newcommand{\N}{\mathbb{N}}
\newcommand{\Z}{\mathbb{Z}}
\newcommand{\CC}{\mathbb{C}}

\newcommand{\eps}{\varepsilon}
\renewcommand{\Re}{\operatorname{Re}}
\renewcommand{\Im}{\operatorname{Im}}

\newcommand\mydots{\makebox[1em][c]{.\hfil.\hfil.}}

\newcommand{\overbar}[1]{\mkern 1.5mu\overline{\mkern-1.5mu#1\mkern-1.5mu}\mkern 1.5mu}

\newcommand{\del}{\partial}
\newcommand{\delbar}{\overbar{\del}}

\newcommand{\Indicator}{{\mathds{1}}}

\newcommand{\CalD}{\mathcal{D}}

\newcommand{\vertiii}[1]{{\left\vert \kern-0.25ex
                            \left\vert \kern-0.25ex
                              \left\vert #1\right\vert\kern-0.25ex
                            \right\vert \kern-0.25ex
                          \right\vert}}
\newcommand{\FirstN}[1]{\underline{#1}}

\newcommand{\loc}{\mathrm{loc}}

\binoppenalty=\maxdimen
\relpenalty=\maxdimen
\makeatletter
\def\blfootnote{\xdef\@thefnmark{}\@footnotetext}
\def\subjclass{\xdef\@thefnmark{}\@footnotetext}
\long\def\symbolfootnote[#1]#2{\begingroup%
\def\thefootnote{\fnsymbol{footnote}}\footnote[#1]{#2}\endgroup}
\if@titlepage
  \renewenvironment{abstract}{%
      \titlepage
      \null\vfil
      \@beginparpenalty\@lowpenalty
      \begin{center}%
        \bfseries \abstractname
        \@endparpenalty\@M
      \end{center}}%
     {\par\vfil\null\endtitlepage}
\else
  \renewenvironment{abstract}{%
      \if@twocolumn
        \section*{\abstractname}%
      \else
        \small
        \list{}{%
          \settowidth{\labelwidth}{\textbf{\abstractname:}}
          \setlength{\leftmargin}{50pt}
          \setlength{\rightmargin}{50pt}
          \setlength{\itemindent}{\labelwidth}
          \addtolength{\itemindent}{\labelsep}
        }
        \item[\textbf{\abstractname:}]

      \fi}
      {\if@twocolumn\else\endlist\fi}
\fi
\makeatother

\let\emptyset\varnothing

\begin{document}

\begin{NoHyper}
\blfootnote{{\it Address:}
Department of Mathematics, Technical University of Munich, 85748 Garching bei München, Germany
\emph{and} Faculty of Mathematics, University of Vienna, Oskar-Morgenstern-Platz~1, 1090~Vienna, Austria.}
\blfootnote{{\it Current Address:}
Mathematical Institute for Machine Learning and Data Science (MIDS),
Catholic University of Eichstätt-Ingolstadt (KU),
Auf der Schanz 49,
85049 Ingolstadt,
Germany}
\blfootnote{{\it Email address:} \texttt{felix.voigtlaender@ku.de}}
\end{NoHyper}

\title{The universal approximation theorem\\{} for complex-valued neural networks}
\author{Felix Voigtlaender}

\date{}

\maketitle

\begin{abstract}
  We generalize the classical universal approximation theorem for neural networks
  to the case of \emph{complex-valued} neural networks.
  Precisely, we consider feedforward networks with a complex activation function
  $\sigma : \CC \to \CC$ in which each neuron performs
  the operation $\CC^N \to \CC, z \mapsto \sigma(b + w^T z)$ with weights $w \in \CC^N$
  and a bias $b \in \CC$, and with $\sigma$ applied componentwise.
  We \emph{completely characterize} those activation functions $\sigma$ for which the associated
  complex networks have the universal approximation property, meaning that they
  can uniformly approximate any continuous function on any compact subset of $\CC^d$
  arbitrarily well.
  Unlike the classical case of real networks, the set of ``good activation functions''---%
  which give rise to networks with the universal approximation property---%
  differs significantly depending on whether one considers deep networks or shallow networks:
  For deep networks with at least two hidden layers, the universal approximation property holds
  as long as $\sigma$ is neither a polynomial, a holomorphic function, or an antiholomorphic function.
  Shallow networks, on the other hand, are universal if and only if
  the real part or the imaginary part of $\sigma$ is \emph{not} a polyharmonic function.
\end{abstract}



\section{Introduction}%
\label{sec:Introduction}

Deep neural networks form the backbone for the recent success \cite{LeCunDeepLearningNature}
of \emph{deep learning} in applications like image recognition \cite{KrizhevskyImagenet}
and language translation \cite{SutskeverMachineTranslation}.
The currently employed networks are mostly \emph{real-valued},
meaning that the network weights and the activation function $\sigma : \R \to \R$ are real.
But for applications in which the inputs are naturally complex-valued---for instance
for MRI fingerprinting \cite{LustigComplexNNForMRI}---complex-valued networks (CVNNs)
\cite{ArenaNNInMultidimensionalDomains} are a natural choice
and often perform better than their real-valued counterparts
\cite{LustigComplexNNForMRI,WolterComplexGatedRNNs,TrabelsiDeepComplexNetworks}.

In such a complex-valued neural network, each neuron computes a function of the form
${\CC^N \ni z \mapsto \sigma(b + w^T z)}$ with a complex bias $b \in \CC$
and complex weights $w \in \CC^N$, and a complex-valued activation function
$\sigma : \CC \to \CC$ that is applied componentwise.
Despite having been somewhat neglected in the wake of the deep learning revolution,
such complex-valued networks are receiving increased attention in recent years
\cite{LustigComplexNNForMRI,WolterComplexGatedRNNs,TrabelsiDeepComplexNetworks,
LeCunMotivationForComplexCNNs}.

The expressivity of real-valued neural networks is by now quite well understood.
Classical results in this area of research are the \emph{universal approximation theorem}
\cite{CybenkoUniversalApproximation,HornikStinchcombeWhiteUniversal,
HornikUniversalApproximation,PinkusUniversalApproximation}%
---stating that sufficiently wide (shallow) neural networks can approximate any (continuous)
function---and more quantitative results like the bounds by Mhaskar
\cite{MhaskarSmoothActivationOptimalRate} regarding the approximation of $C^n$
functions using shallow networks with smooth activation functions.
More recent research \cite{YarotskyTraditionalReLUBounds,PetersenOptimalApproximation,
YarotskyPhaseDiagram,LuReLUFiniteDepthUniform}
has focused on similar quantitative results for \emph{deep} networks
with \emph{non-smooth} activation functions, in particular the ReLU $\varrho(z) = \max \{ 0, z \}$.
This is because deep networks perform better in applications \cite{LeCunDeepLearningNature},
and the ReLU has been shown to lead to faster training \cite{GlorotRectifierNetworks}.

Concerning the approximation properties of complex-valued networks,
on the other hand, the literature is quite scarce.
To the best of our knowledge, the only available results are purely qualitative versions
of the universal approximation theorem, and even these are only available
for very special activation functions;
see \cite{ArenaApproximationCapabilityOfComplexNeuralNetworks,ArenaMLPToApproximateComplexValuedFunctions}
and \cite[Chapter~2]{ArenaNNInMultidimensionalDomains},
as well as \cite[Chapter~5]{HiroseCliffordBook}.
We refer to \Cref{sub:RelatedWork} for a more detailed description
of these and other related articles.

Our goal in this article is to initiate a more in-depth study of the approximation properties
of these networks, by providing a \emph{full characterization} of those complex
activation functions $\sigma : \CC \to \CC$ which lead to network classes that are \emph{universal},
in the sense that, for a \emph{fixed} network depth $L \in \N$,
\emph{every} continuous function $f : \CC^d \to \CC$ can be approximated
arbitrarily well (uniformly on compact sets) by sufficiently wide complex-valued
neural networks of depth $L$ using the activation function $\sigma$.
In other words, the question is whether the set $\NN_{\sigma,L}^d$
of all (arbitrarily wide) complex-valued neural networks
with $L$ hidden layers and activation function $\sigma$ is dense
in $C(\CC^d;\CC)$ in the topology of locally uniform convergence.
For the case of \emph{real-valued} networks, such a characterization is surprisingly simple
\cite{PinkusUniversalApproximation}: these networks are universal if and only if the
activation function $\sigma : \R \to \R$ does not coincide with a polynomial (almost everywhere).
This characterization is quite natural---at least as a necessary condition---since if $\sigma$
is a polynomial of degree $N$, then each function implemented by a network with $L$ hidden layers
and activation $\sigma$ is a polynomial of degree at most $N^L$,
\emph{irrespective of the width of the network}.
This clearly rules out universality, since $L$ is considered fixed%
\footnote{
  As a side-note, we remark that if one considers real-valued neural networks
  of \emph{fixed (but sufficiently large) width $W \geq W_0(d)$}
  but \emph{of arbitrary depth $L$}, then (essentially)
  every \emph{non-affine} activation function $\sigma : \R \to \R$
  (including non-affine \emph{polynomials}) leads to a universal class of networks;
  this was shown in \cite{DeepNarrowNetworksUniversalApproximation}.
}.

But as already noted in \cite{ArenaApproximationCapabilityOfComplexNeuralNetworks},
in the complex domain more subtle obstructions to universality appear.
In fact, the results in \cite{ArenaApproximationCapabilityOfComplexNeuralNetworks}
show that a \emph{holomorphic} (entire) activation function can never give rise
to universal sets of networks.
Some authors (see for instance \cite{KimApproximationWithFullyComplexNN}) have tried to circumvent
this obstruction by using holomorphic functions \emph{with singularities}
as activation functions.
However, even this approach does not succeed in general, as shown in the following result.
For a detailed discussion of the findings in \cite{KimApproximationWithFullyComplexNN},
we refer to \Cref{sub:RelatedWork}.

\begin{proposition}\label{prop:IntroHolomorphicIsBad}
  Let $D \subset \CC$ be closed and discrete, and let $\sigma : \CC \setminus D \to \CC$ be holomorphic.
  A weight-tuple $\Theta = \big( (\alpha_1,w_1,b_1),\dots,(\alpha_N,w_N,b_N) \big) \in (\CC^3)^N$
  is called \emph{admissible} if for each $j \in \{ 1,\dots,N \}$
  we have $w_j \neq 0$ or $b_j \notin D$.
  In this case, define
  \[
    D_\Theta := \bigcup_{j=1}^N \bigl\{ z \in \CC \colon w_j z + b_j \in D \bigr\}
    \qquad \text{and} \qquad
    \Phi_\Theta : \quad
    \CC \setminus D_\Theta \to \CC, \quad
    z \mapsto \sum_{j=1}^N \alpha_j \, \sigma(b_j + w_j z) .
  \]
  There exists some $\eps > 0$ and a function $f \in C_c(\CC; \CC)$ (both independent of $N \in \N$)
  satisfying
  \[
    \inf_{\substack{N \in \N, \\ \Theta \in (\CC^3)^N \text{ admissible}}} \quad
      \sup_{z \in B_1(0) \setminus D_\Theta} \quad
        \bigl|f(z) - \Phi_\Theta (z)\bigr|
    \geq \eps
    \qquad \text{where} \qquad
    B_1(0) = \bigl\{ z \in \CC \colon |z| < 1 \bigr\} .
  \]
\end{proposition}

\begin{rem*}
  The assumption of admissibility ensures that the set $D_\Theta$ on which
  $\Phi_\Theta$ is not well-defined is discrete.
\end{rem*}

\begin{proof}
  See \Cref{sub:ProofOfHolomorphicIsBad}.
\end{proof}

\begin{remark}\label{rem:MostlyHolomorphicCanBeUniversal}
  \Cref{prop:IntroHolomorphicIsBad} only considers holomorphic functions with \emph{isolated}
  singularities; the proof crucially uses that such a singularity is either removable,
  or that otherwise the function is \emph{unbounded} on each neighborhood of the singularity.

  Still, there are \emph{universal} activation functions $\sigma : \CC \to \CC$ that are holomorphic
  on \emph{most} of their domain.
  To give an example, let $\CC_- = \CC \setminus (-\infty,0]$, and denote by
  $\operatorname{Log} : \CC_- \to \CC$ the principal branch of the complex logarithm,
  given by $\operatorname{Log} (r \, e^{i \theta}) = \ln(r) + i \theta$, where $r > 0$
  and $\theta \in (-\pi,\pi)$.
  Then the function
  \[
    \sigma : \quad
    \CC \to \CC, \quad
    z \mapsto \begin{cases}
                z \cdot \operatorname{Log}(z), & \text{if } z \in \CC_-,  \\
                0                              & \text{otherwise}
              \end{cases}
  \]
  is holomorphic on $\CC_-$.
  In particular, this implies that the closure of the set of discontinuities of $\sigma$
  is a subset of $(-\infty,0]$, and hence a null-set when considered as a subset of $\CC$.
  Furthermore, $\sigma$ is locally bounded (thanks to the factor $z$ compensating the blow-up
  of the logarithm at the origin), and it is easy to see that there does \emph{not}
  exist a smooth function ${g : \CC \to \CC}$ satisfying $\sigma = g$ almost everywhere.
  Therefore, \Cref{thm:ShallowUniversalApproximation} below shows that the set of shallow
  complex-valued neural networks with activation function $\sigma$ has the universal approximation
  property.
\end{remark}

\Cref{prop:IntroHolomorphicIsBad} shows that the question of ``universal activation functions''
is much more subtle in the complex domain than for real networks.
Nevertheless, we explicitly characterize those activation functions $\sigma : \CC \to \CC$
for which the associated sets of networks have the universal approximation property;
see the next subsection.
Surprisingly, the class of ``good'' activation functions for \emph{shallow} networks
differs from the class of activation functions that yield universality for \emph{deep} networks.
This is again in stark contrast to the real-valued case.

\subsection{Our results in a nutshell}%
\label{sub:IntroductionResults}

In order to state our results concisely, we first introduce a few notations and conventions.
Similar to \cite{PinkusUniversalApproximation}, we restrict our attention to activation functions
satisfying a minimal continuity property.
Precisely, let us write $\CalM$ for the class of all functions $\sigma : \CC \to \CC$
that are locally bounded and such that \emph{the closure} of the set
$\{ z \in \CC \,\colon\, \sigma \text{ not continuous at } z \}$
is a null-set, in the sense of the Lebesgue measure on $\CC \cong \R^2$.

We do \emph{not} identify functions that agree almost everywhere,
since we will need to consider compositions of such functions,
and it is \emph{not} true in general that $f_1 \circ g_1 = f_2 \circ g_2$ almost everywhere
if $f_1 = f_2$ and $g_1 = g_2$ almost everywhere.
Given this convention, the following notation is natural:
For $\emptyset \neq \Omega \subset \CC^d$ and $f : \Omega \to \CC$, we write
\begin{equation}
  \| f \|_{L^\infty(\Omega)}
  := \sup_{z \in \Omega}
       |f(z)| .
  \label{eq:SpecialLInftyNorm}
\end{equation}
We say that a set $\CalF \subset \{ f : \CC^d \to \CC \}$ has the
\emph{universal approximation property} if for every continuous function $\varphi : \CC^d \to \CC$,
every compact set $K \subset \CC^d$ and every $\eps > 0$ there exists $f \in \CalF$
satisfying $\sup_{z \in K} |\varphi(z) - f(z)| \leq \eps$.

Finally, unless explicitly mentioned otherwise, smoothness and differentiability are always
understood in the sense of real variables (and \emph{not} in the sense of holomorphic functions),
under the usual identification $\CC^d \cong \R^{2 d}$.
Then, we say that a function $\sigma : \CC \to \CC$ is \emph{almost polyharmonic} if there exist
$m \in \N$ and ${g \in C^\infty(\CC; \CC)}$ with $\Delta^m g \equiv 0$
such that $\sigma = g$ almost everywhere.
Here, $\Delta = \frac{\partial^2}{\partial x^2} + \frac{\partial^2}{\partial y^2}$
is the usual Laplace operator on $\CC \cong \R^2$.

Our first main results provides a complete characterization of those activation functions
for which the associated class $\NN_\sigma^d \subset \{ f : \CC^d \to \CC \}$ of
\emph{shallow} complex-valued neural networks with activation function $\sigma$, given by
\[
  \NN_\sigma^d
  := \!
     \Big\{
       z \mapsto c +
                 \sum_{j=1}^N
                   a_j \, \sigma\bigl(b_j + w_j^T z\bigr)
       \,\, \colon \,\,
       N \in \N, \,
       a_1, b_1, \mydots, a_N, b_N, c \in \CC
       \text{ and } w_1, \mydots, w_N \in \CC^d
     \Big\} ,
\]
is universal.

\begin{theorem}\label{thm:ShallowUniversalApproximation}
  Let $\sigma \in \CalM$ and $d \in \N$ be arbitrary.
  Then the set ${\NN_\sigma^d \subset \{ f : \CC^d \to \CC \}}$ of shallow complex-valued
  neural networks with activation function $\sigma$ has the universal approximation property
  \emph{if and only if} $\sigma$ is \emph{not} almost polyharmonic.
\end{theorem}

\begin{rem*}
  The theorem implies that the class of ``good'' activation functions is quite rich.
  For example, it is well-known (see for instance \cite[Remark~2]{HuilgolBiharmonicLiouville})
  that bounded polyharmonic functions are necessarily constant.
  Thus, the theorem shows that if $\sigma \in \CalM$ is bounded but not equal to a constant
  (almost everywhere), then $\NN_{\sigma}^d$ is universal.

  Furthermore, the theorem shows that any ``bad'' activation function $\sigma$
  necessarily has to be smooth (possibly after changing it on a null-set).
  Thus, any activation function that can not be made smooth by changing it on a null-set
  is a ``universal'' activation function.
\end{rem*}

Let us denote by $\NN_{\sigma,L}^d$ the set of complex-valued feedforward neural networks
with activation function $\sigma$, input dimension $d$, and $L$ hidden layers;
a precise definition is given in \Cref{def:DeepComplexNetworks} below.
The following theorem characterizes the activation functions $\sigma$
for which the deep network class $\NN_{\sigma,L}^d$ ($L \geq 2$) is universal.
Contrary to the real setting, there are strictly more ``universal'' activation functions
for deep networks than for shallow ones.

\begin{theorem}\label{thm:DeepUniversalApproximationIntroduction}
  Let $\sigma \in \CalM$ and assume that \emph{none} of the following properties hold:
  \begin{enumerate}[label=\alph*)]
    \item we have $\sigma (z) = p(z, \overline{z})$ for almost all $z \in \CC$,
          where $p \in \CC[X,Y]$ is a complex polynomial of two variables,

    \item we have $\sigma = g$ almost everywhere or $\sigma = \overline{g}$ almost everywhere,
          where $g : \CC \to \CC$ is an entire holomorphic function.
  \end{enumerate}
  Then, for each $L \in \N_{\geq 2}$ and each $d \in \N$, the class $\NN_{\sigma,L}^d$
  of deep complex-valued neural networks with activation function $\sigma$ and $L$ hidden layers
  has the universal approximation property.

  Conversely, if $\sigma : \CC \to \CC$ is continuous and satisfies a) or b),
  then $\NN_{\sigma,L}^d$ does \emph{not} satisfy the universal approximation property
  for any $d, L \in \N$.
\end{theorem}

\begin{rem*}
  Note that the necessary condition requires $\sigma$ to be continuous,
  whereas for the sufficient condition it is only assumed that $\sigma \in \CalM$.
  The additional continuity assumption is not simply a proof artifact.
  In fact, we will see in \Cref{exa:PathologicalUniversalSigma} that there does exist
  a discontinuous activation function $\sigma \in \CalM$ that coincides almost everywhere
  with a polynomial $p(z, \overline{z})$, but such that $\NN_{\sigma,L}^d$ is
  nevertheless universal for all $L \in \N_{\geq 2}$ (but not for $L = 1$).
  We leave it as future work to determine natural conditions on $\sigma$
  that are weaker than continuity, but under which a necessary condition
  as in the above theorem still holds.

  We would like to emphasize that any \emph{discontinuous} function $\sigma$
  satisfying (at least) one of the two conditions
  in \Cref{thm:DeepUniversalApproximationIntroduction} has to be quite pathological:
  Such a function is discontinuous, but coincides almost everywhere with a continuous function.
  In that case, it is usually much more natural to work with the ``continuous version''
  of $\sigma$ instead of with $\sigma$ itself.
  Moreover, such a function cannot be evaluated numerically in a reliable way
  and is therefore useless for practical applications.
\end{rem*}

\begin{remark}\label{rem:NecessityEvenL1}
  In both theorems above, it is only stated that if the conditions of the theorem fail,
  then the universal approximation property does not hold, meaning that some continuous
  function cannot be approximated by networks \emph{in the sense of locally uniform convergence}.
  In fact, the proof even shows that some \emph{compactly supported} continuous function
  cannot be approximated in $L_{\loc}^1 (\CC^d)$;
  see \Cref{thm:ShallowUniversalApproximationNecessary,thm:DeepNetworksNecessaryCriterion}.
\end{remark}

\subsection{Related work}%
\label{sub:RelatedWork}

\paragraph{The classical universal approximation theorem}

There exist many versions of the universal approximation theorem for \emph{real} networks.
One of the first versions of this theorem is due to Cybenko \cite{CybenkoUniversalApproximation},
who introduced the notion of \emph{discriminatory functions} to prove his result;
see \cite[Section~2]{CybenkoUniversalApproximation}.
By standard properties in functional analysis, it follows that the set of shallow networks
using a discriminatory activation function has the universal approximation property.
What is more challenging is to prove that certain activation functions are indeed discriminatory.
Cybenko verified this for continuous functions $\sigma : \R \to \R$ satisfying
$\lim_{x \to -\infty} \sigma(x) = 0$ and $\lim_{x \to \infty} \sigma(x) = 1$;
such functions are called \emph{sigmoidal}.
A slightly modified version of this result was shown in \cite{HornikStinchcombeWhiteUniversal},
where universality was established for possibly discontinuous, non-decreasing sigmoidal
activation functions.
Hornik \cite{HornikUniversalApproximation} proved universality for arbitrary non-constant,
bounded, continuous activation functions.
Finally, dropping the assumptions of continuity, monotonicity, boundedness, and sigmoidality,
the results in \cite{PinkusUniversalApproximation} established universality
for \emph{every} non-polynomial activation function, under some minimal continuity assumptions.
More precisely, \cite{PinkusUniversalApproximation} considers locally bounded functions
$\sigma : \R \to \R$ for which the closure of the set of discontinuities of $\sigma$ is a null-set;
the class $\CalM$ of complex-valued functions that we consider is the natural
generalization of this class to the complex domain.

\paragraph{Alternative network architectures}

The classical universal approximation theorem concerns fully connected
feedforward networks of a fixed depth and arbitrary width.
Recently, several alternative network architectures have been studied as well.
For instance, \cite{LuUniversalApproximationBoundedWidth} shows that ReLU networks
with bounded width $d+4$ but \emph{depth tending to infinity} are universal approximators.
Furthermore, \cite{ResNetUniversalApproximation} shows that so-called residual networks
with one neuron per hidden layer but arbitrarily many layers are universal as well.
The universality of \emph{convolutional networks} has been studied in
\cite{ZhouCNNUniversality,ZhouCNNUniversalityDownsampling,
PetersenVoigtlaenderConvolutionalNNEquivalence,YarotskyInvariantUniversal}.
In \cite{ZhouCNNUniversality,ZhouCNNUniversalityDownsampling}, it is shown that
the usual finite-dimensional convolutional networks using many layers and
zero-padded convolutions with short filters---either with \cite{ZhouCNNUniversalityDownsampling}
or without \cite{ZhouCNNUniversality} subsampling---are universal.
The case of convolutions with periodic boundary conditions and without subsampling
is treated in \cite{PetersenVoigtlaenderConvolutionalNNEquivalence}.
Finally, Yarotsky \cite{YarotskyInvariantUniversal} studied the universality
of convolutional networks in a continuous, infinite-dimensional setting.

\paragraph{Universal approximation for complex-valued neural networks}

Some of the first articles studying the approximation properties of complex-valued neural networks
are the papers
\cite{ArenaApproximationCapabilityOfComplexNeuralNetworks,ArenaMLPToApproximateComplexValuedFunctions}
by Arena et al.
In these articles, it was observed that complex-valued networks with holomorphic activation functions
cannot be universal.
Furthermore, \cite{ArenaMLPToApproximateComplexValuedFunctions} generalized the notion
of discriminatory functions introduced in \cite{CybenkoUniversalApproximation}
to the complex domain, showing that networks with discriminatory activation functions are universal.
Again, the actual difficulty lies in verifying that a given activation function is discriminatory.
This was shown in
\cite{ArenaApproximationCapabilityOfComplexNeuralNetworks,ArenaMLPToApproximateComplexValuedFunctions}
for the function $\sigma(z) = \frac{1}{1 + \exp(-\Re z)} + \frac{i}{1 + \exp(-\Im z)}$.
The same result appears in a somewhat more accessible form in
\cite[Theorem~2.6.3]{ArenaNNInMultidimensionalDomains}.
Another activation function that has been shown to be discriminatory is $\sigma(z) = \frac{z}{1 + |z|}$.
This fact can be found in \cite[Chapter~5]{HiroseCliffordBook}, where more general
Clifford-algebra-valued networks are considered; these specialize to complex-valued networks.

Two other widely cited works on the approximation properties of complex-valued neural networks
are \cite{KimApproximationWithFullyComplexNN} and \cite{HuangExtremeLearningFullyComplex}.
Both of these articles, however, are not entirely correct.
In \cite{KimApproximationWithFullyComplexNN}, several holomorphic activation functions
with singularities are considered.
Furthermore, even the entire functions $\sin$ and $\sinh$ are proposed as activation functions.
It is then claimed in the abstract and the summary of \cite{KimApproximationWithFullyComplexNN}
that the neural network sets associated to these activation functions are universal,
in the sense that they can uniformly approximate arbitrary continuous functions
on compact subsets of $\CC^n$.
However, as shown in \Cref{prop:IntroHolomorphicIsBad}, this is false in general.
More precisely, one can classify the activation functions
proposed in \cite{KimApproximationWithFullyComplexNN} as follows:
\begin{enumerate}
  \item Entire functions ($\sin$ and $\sinh$), or holomorphic functions with only isolated singularities
        ($\tan$ and $\tanh$).
        For these activation functions, \Cref{prop:IntroHolomorphicIsBad} shows that the associated
        neural networks are \emph{not} universal (even if one redefines them in an arbitrary
        way at the points of the singularities),
        contrary to what is claimed in \cite{KimApproximationWithFullyComplexNN}.

  \item Functions whose definition involves a branch cut,
        but which are \emph{locally bounded} on $\CC$.
        These are $\arcsin, \arccos$, and $\arcsinh$.
        If one considers (for instance) the principal branches of these functions,
        then they are locally bounded on $\CC$ and the closure of their set of discontinuities
        forms a null-set.
        Further, it is easy to see that \emph{none} of these functions coincides almost everywhere
        with a smooth function.
        Hence, \Cref{thm:ShallowUniversalApproximation} shows that shallow complex-valued
        neural networks with these activation functions are indeed universal.

  \item Function whose definition involves a branch cut,
        but which are \emph{not} locally bounded on $\CC$.
        The activation functions of this type considered in \cite{KimApproximationWithFullyComplexNN}
        are $\arctan$ and $\arctanh$.
        For these functions, the theory developed in the present article is not sufficient
        to decide whether the associated complex-valued neural networks are universal or not.
        However, precisely because these activation functions are not locally bounded,
        it is questionable whether they will be useful in practice.
\end{enumerate}
We emphasize that the proofs given in \cite{KimApproximationWithFullyComplexNN}
do \emph{not} correctly prove the universality for the activation functions mentioned in
Points 2 and 3 above, as the very same proofs also (incorrectly)
show universality for the activation functions mentioned in the first point.

Furthermore, there is a mismatch between what is stated in the abstract and summary
of \cite{KimApproximationWithFullyComplexNN} and what is actually stated in the theorems.
For instance, \cite[Theorems~1-3]{KimApproximationWithFullyComplexNN} concern the density
of functions of the form
$\sum_{k=1}^m \beta_k \prod_{\ell=1}^{s_k} \sigma(\theta_{k,\ell} + w_{k,\ell}^T z)$,
which are \emph{not} shallow neural networks with activation function $\sigma$,
due to the product appearing in the definition.
Additionally, \cite[Theorems~1 and 2]{KimApproximationWithFullyComplexNN} consider the approximation
of functions defined on ${I_n = [0,1]^n}$, which has empty interior as a subset of $\CC^n$.
For continuous functions on such domains, universal approximation by complex networks
with holomorphic activation functions \emph{is actually possible}, but this is not the
problem considered in the present article and also different from what is claimed
in the abstract and summary of \cite{KimApproximationWithFullyComplexNN}.

Finally, the article \cite{HuangExtremeLearningFullyComplex} considers so-called
\emph{extreme learning machines} with complex weights and activation functions.
These are essentially deep networks in which the weights of the ``lower layers''
are chosen randomly (according to a distribution with full support),
and only the weights on the output layer can be adjusted
in order to achieve the desired approximation.
In the real setting, it has been shown that such extreme learning machines
can almost surely approximate a given continuous function
if more and more random hidden neurons are added \cite{HuangRealELMUniversality};
the proof essentially uses the classical universal approximation in \cite{PinkusUniversalApproximation}
as a black box and shows that the statement remains true if one uses randomly chosen weights,
except on the last layer.
In \cite{HuangExtremeLearningFullyComplex}, a similar reasoning is used to argue that an
analogous result holds in the complex domain.
But for doing so, the authors cite the incorrect results from \cite{KimApproximationWithFullyComplexNN}.
Furthermore, due to the issues in \cite{KimApproximationWithFullyComplexNN} discussed above,
each ``neuron'' in the extreme learning machines considered in
\cite{HuangExtremeLearningFullyComplex} is of the form
$\prod_{\ell=1}^{s_k} \sigma(\theta_{k,\ell} + w_{k,\ell}^T z)$.
Due to the product, this disagrees with the usual definition of neural networks.
It should also be noted that no proof is given for \cite[Lemma~2.6]{HuangExtremeLearningFullyComplex},
which is an essential ingredient for the proofs in \cite{HuangExtremeLearningFullyComplex}.

In summary, to the best of our knowledge, the present article shows for the first time
that a wide class of activation functions leads to complex-valued neural networks
with the universal approximation property.
In fact, we provide a complete characterization of such activation functions,
subject to very mild local boundedness and continuity assumptions.

\subsection{Structure of this article}%
\label{sub:Structure}

After introducing some general notation, \Cref{sec:WirtingerCalculus} gives a brief
introduction to the \emph{Wirtinger calculus}, which will be essential for all that follows.
\Cref{sec:ProofSketch} paves the way for the formal and somewhat technical proofs
in \Cref{sec:Proofs}, by presenting the gist of the argument in a simplified form.
The presentation of the formal proofs in \Cref{sec:Proofs} is split into four subsections.
Considering the case of shallow networks, the first two subsections provide
a proof of \Cref{thm:ShallowUniversalApproximation}, split into separate proofs
for sufficiency and necessity.
Similarly, the last two subsections give a detailed proof of
\Cref{thm:DeepUniversalApproximationIntroduction}.
Several more technical results are deferred to the appendix.

\subsection{Notation}%
\label{sub:notation}

Throughout this article, we always consider $\CC$ as a vector space over $\R$,
meaning that differentiability of a function $f : U \subset \CC^d \to \CC$ will refer to
\emph{real differentiability}, not differentiability in the sense of holomorphic functions,
unless explicitly mentioned otherwise.
In particular, in the context of partial derivatives
$\frac{\partial}{\partial a}, \frac{\partial}{\partial x}$, etc.,
the variables $a,x,\dots$ should always be considered as real variables.
For a function $f : U \subset \R^d \to \CC$, we will use the notation $\partial_j f$
for the partial derivative of $f$ with respect to the $j$-th variable.
The same notation is used for $f : U \subset \CC \to \CC$,
where $f$ is identified with the real function $(x,y) \mapsto f(x + i y)$.
The sets of continuous, $n$-times continuously differentiable, or smooth complex-valued
functions on an open set $\emptyset \neq U \subset \CC^d$ are denoted by
$C(U; \CC)$, $C^n (U; \CC)$, and $C^\infty( U; \CC)$, respectively.
The set of smooth complex-valued functions with compact support in $U$ is denoted by
$C_c^\infty(U ; \CC)$.

On $\CC^d$, we use the Euclidean norm $|z| = \bigl(\sum_{j=1}^d |z_j|^2\bigr)^{1/2}$
as well as the maximum norm $\| z \|_{\infty} = \max_{j = 1,\dots,d} |z_j|$.
Open and closed Euclidean balls are denoted by
\[
  B_r (z)
  = \{
      w \in \CC^d
      \colon
      |w - z| < r
    \}
  \qquad \text{and} \qquad
  \overline{B_r} (z)
  = \{
      w \in \CC^d
      \colon
      |w - z|
      \leq r
    \} ,
\]
respectively; here, $r > 0$ and $z \in \CC^d$.
For a subset $D \subset \CC^d$, we will use the notation $\overline{D}$ to denote the closure of $D$.
A similar notation $\overbar{z}$ is used to denote the conjugate of a complex number $z \in \CC$;
the context will make clear which interpretation is intended.

\medskip{}

It is usually not possible to approximate all continuous functions
\emph{uniformly on all of $\R^d$ or $\CC^d$} using neural networks;
therefore, we will focus on \emph{locally uniform convergence}.
To be precise, on the set $\mathrm{Map}(\CC^d; \CC) = \{ f : \CC^d \to \CC \}$,
we consider the \emph{topology of compact convergence}
(also called \emph{the topology of uniform convergence on compact sets}),
as defined in \cite[§46]{MunkresTopology}.
This topology is generated by the sets
\[
  B_K (f,\eps)
  := \Big\{
       g : \CC^d \to \CC
       \,\,\colon\,\,
       \sup_{x \in K} |g(x) - f(x)| < \eps
     \Big\}
  ,
\]
where $f : \CC^d \to \CC$ and $\eps > 0$ are arbitrary, and where $K \subset \CC^d$
is any compact set.
In fact, the collection of all the sets $B_K (f,\eps)$ forms a \emph{basis}
for the topology of compact convergence; see \cite[§46]{MunkresTopology}.
Thus, given a set $\CalF \subset \{ f : \CC^d \to \CC \}$, the closure of $\CalF$
with respect to this topology is given by
\begin{equation}
  \overline{\CalF}
  := \big\{
       g : \CC^d \to \CC
       \,\,\colon\,\,
       \forall \, K \subset \CC^d \text{ compact and } \eps > 0 \quad
         \exists \, f \in \CalF : \quad
           \| g - f \|_{L^\infty (K)} \leq \eps
     \big\}
  ,
  \label{eq:ClosureDefinition}
\end{equation}
where we recall from \Cref{eq:SpecialLInftyNorm}
that $\| g \|_{L^\infty (K)} = \sup_{x \in K} |g(x)|$.

An important observation that we will use again and again is
that if $g \in \overline{\CalF}$, then there exists a \emph{sequence} $(f_n)_{n \in \N} \subset \CalF$
satisfying $f_n \to g$ locally uniformly.
Indeed, for each $n \in \N$ we can choose $f_n \in \CalF$ satisfying
$\| g - f_n \|_{L^\infty(\overline{B_n}(0))} \leq n^{-1}$,
and then clearly $f_n \to g$ locally uniformly.
The converse is also easily seen to hold (i.e., if $(f_n)_{n \in \N} \subset \CalF$
satisfies $f_n \to g$ locally uniformly, then $g \in \overline{\CalF}$).
Therefore, if one is working on a locally compact space like $\CC^d$,
this topology is also often called the \emph{topology of locally uniform convergence}.
In particular, this characterization of the closure $\overline{\CalF}$ implies
that the set $C(\CC^d;\CC)$ of continuous functions and the set
$\mathrm{Meas}_{\mathcal{B}}(\CC^d; \CC) = \{ f : \CC^d \to \CC \colon f \text{ is Borel-measurable} \}$
of Borel-measurable functions, as well as the set
$\mathrm{Meas}_{\mathcal{L}}(\CC^d; \CC) = \{ f : \CC^d \to \CC \colon f \text{ is Lebesgue-measurable} \}$
of Lebesgue-measurable functions are all closed in $\mathrm{Map}(\CC^d;\CC)$ with respect to this topology.
Unless mentioned otherwise, we will always understand the closure $\overline{\CalF}$
of a set $\CalF$ of functions to be defined as in \Cref{eq:ClosureDefinition}.
Also recall once more that we \emph{do not} identify functions that agree almost everywhere.

\smallskip{}

We frequently use the notation $\FirstN{n} := \{ k \in \Z \colon 1 \leq k \leq n \}$
for $n \in \N_0 = \{ 0,1,2,\dots \}$; in particular, $\FirstN{0} = \emptyset$.
To signify a \emph{disjoint} union, we write $\biguplus_{j \in I} B_j$.
The Lebesgue measure on $\R^d$ will be denoted by $\lambda_d$ or simply by $\lambda$.
Identifying $\CC^d \cong \R^{2 d}$, we also consider $\lambda_{2 d}$ as a measure on $\CC^d$.
Finally, given a function $\sigma : \CC \to \CC$, we extend it to a map on $\CC^d$
by acting componentwise: $\sigma(z_1, \dots, z_d) = \bigl(\sigma(z_1),\dots,\sigma(z_d)\bigr)$.


\section{The Wirtinger calculus and polyharmonic functions}%
\label{sec:WirtingerCalculus}

We will make heavy use of the so-called \emph{Wirtinger calculus};
\cite[§1]{KaupHolomorphicFunctionsSeveralVariables}.
Precisely, given an open set $\emptyset \neq U \subset \CC$ and a $C^1$ function $f : U \to \CC$,
we define the \emph{Wirtinger derivatives}
\[
  \del f := \frac{1}{2} \big( \partial_1 f - i \, \partial_2 f \big)
  \qquad \text{and} \qquad
  \delbar f := \frac{1}{2} \big( \partial_1 f + i \, \partial_2 f \big) .
\]
In the case of functions $f = f(z,w)$ of several complex variables, we will use the notation
$\del_w$ and $\delbar_w$ to specify the variable with respect
to which the Wirtinger derivatives are taken.

\begin{rem*}
  In the usual literature, the notation $\frac{\partial}{\partial z}$
  and $\frac{\partial}{\partial \overline{z}}$ is used instead of $\del$
  and $\delbar$.
  We refrain from using this notation to avoid ambiguities with the notation
  for partial derivatives in the sense of real differentiability.
\end{rem*}

We will use the following properties of the Wirtinger derivatives, given e.g.~in
\cite[E.~1a]{KaupHolomorphicFunctionsSeveralVariables}:
\begin{itemize}
  \item $\del$ and $\delbar$ are $\CC$-linear;

  \item a function $f \in C^1(U; \CC)$ is holomorphic if and only if $\delbar f \equiv 0$,
        and in this case the complex derivative $f'$ of $f$ coincides with $\del f$;

  \item the Wirtinger derivatives are compatible with conjugation in the sense that
        \begin{equation}
          \overline{\delbar f} = \del \overline{f}
          \quad \text{and} \quad
          \overline{\del f} = \delbar \, \overline{f} ;
          \label{eq:WirtingerConjugation}
        \end{equation}

  \item the product rule holds, meaning
        \[
          \del (f \cdot g) = g \cdot \del f + f \cdot \del g
          \quad \text{and} \quad
          \delbar (f \cdot g)
          = g \cdot \delbar f + f \cdot \delbar g
          \quad \text{for} \quad f,g \in C^1(U; \CC);
        \]

  \item we have the following variant of the chain rule
        for $g \in C^1(U; V)$ and $f \in C^1(V; \CC)$
        with open sets $\emptyset \neq U, V \subset \CC$:
        \begin{equation}
          \del (f \circ g)
          = \bigl[(\del f) \circ g\bigr] \cdot \del g
            + \bigl[(\delbar f) \circ g\bigr] \cdot \del \overbar{g}
          \quad \text{and} \quad
          \delbar (f \circ g)
          = \bigl[(\del f) \circ g\bigr] \cdot \delbar g
            + \bigl[(\delbar f) \circ g\bigr] \cdot \delbar \, \overbar{g} .
          \label{eq:ChainRule}
        \end{equation}
\end{itemize}

Because of $\del \overbar{z}^m = \overline{\delbar z^m} = 0$,
the product rule implies for $m \in \N_0$ and $f \in C^1(U; \CC)$ that
\(
  \del\bigl(\overbar{z}^m \cdot f(z)\bigr)
  = f(z) \cdot \del \overbar{z}^m
    + \overbar{z}^m \cdot \del f(z)
  = \overbar{z}^m \cdot \del f(z)
  .
\)
By induction (and separately considering the case $\ell = 0$), this easily shows
\begin{equation}
  \del^\ell \bigl(\overbar{z}^m \cdot f(z)\bigr)
  = \overbar{z}^m \cdot \del^\ell f(z)
  \quad \text{for } \ell, m \in \N_0 \text{ and } f \in C^\ell(U ; \CC).
  \label{eq:SpecialProductRule}
\end{equation}

The Wirtinger calculus is closely related to the class of
\emph{polyharmonic functions}; see for instance \mbox{\cite[Pages 12--13]{BalkPolyanalyticFunctions}}.
Precisely, given an open set $\emptyset \neq U \subset \CC$,
a (real-valued) $C^{2 m}$-function $f : U \to \R$ is called \emph{polyharmonic of order $m$} if
$\Delta^m f \equiv 0$, where $\Delta = \frac{\partial^2}{\partial x^2} + \frac{\partial^2}{\partial y^2}$
is the usual Laplace operator on $\CC \cong \R^2$.
We will say that a complex-valued function $f : U \to \CC$ is (complex) polyharmonic of order $m$
if both $\Re f$ and $\Im f$ are polyharmonic of order $m$.
The relation between polyharmonic functions and the Wirtinger calculus is due to the identity
\begin{equation}
  \Delta f = 4 \, \del \delbar f
  \quad \text{for} \quad
  f \in C^2(U; \CC) ,
  \label{eq:WirtingerLaplaceRepresentation}
\end{equation}
given in \cite[Equation~(1.7)]{BalkPolyanalyticFunctions}.


\section{The main proof ideas}%
\label{sec:ProofSketch}

In this section, we sketch the main proof ideas, to prepare for the
fully rigorous, but somewhat technical proofs in \Cref{sec:Proofs}.
We remark that many ideas are inspired by the techniques in \cite{PinkusUniversalApproximation};
we modify these ideas to apply to the complex domain.

\subsection{The case of shallow networks}%
\label{sub:ShallowProofSketch}

\paragraph{Proving sufficiency for smooth $\sigma$:}

Let us assume that $\sigma$ is smooth but \emph{not} polyharmonic.
Thanks to \Cref{eq:WirtingerLaplaceRepresentation}, this implies
$\del^m \delbar^\ell \sigma \not\equiv 0$ for all $m,\ell \in \N_0$.
A direct calculation using the properties of the Wirtinger calculus shows that
\begin{equation}
  \del_w^m \delbar_w^\ell \big|_{w = 0}
  \bigl[ \sigma(w z + \theta) \bigr]
  = z^m \, \overbar{z}^\ell \cdot (\del^m \delbar^\ell \sigma) (\theta) .
  \label{eq:IntroductionMonomialExtraction}
\end{equation}
By choosing $\theta = \theta_{m,\ell}$ such that $(\del^m \delbar^\ell \sigma) (\theta) \neq 0$,
one can thus ``extract'' the monomial $z^m \, \overbar{z}^\ell$ from certain
Wirtinger derivatives of $\sigma$.

Note that $\Psi_w : z \mapsto \sigma(w z + \theta)$ can be computed by a single neuron.
Furthermore, since the set $\NN_{\sigma}^1$ is a vector space and closed under translations
and dilations, one can compute difference quotients of $w \mapsto \Psi_w (z)$ using shallow networks.
Based on this observation, one can show that the derivative on the left-hand side of
\Cref{eq:IntroductionMonomialExtraction} can be approximated by shallow networks;
see the proofs of \Cref{lem:MonomialExtraction} and \Cref{prop:DerivativesInTheSpace}
for the details.

By the complex version of the \emph{Stone-Weierstraß theorem}, any continuous function
can be uniformly approximated%
\footnote{Note that this is decidedly false if one only considers polynomials in $z$,
since these are always holomorphic, which is preserved under locally uniform convergence.}
by polynomials $p(z,\overbar{z})$.
This proves universality for input dimension $d = 1$.
Generalizing this to higher input dimension is mainly technical;
see \Cref{lem:FromOneDimensionToHigherDimensions}.

\paragraph{Generalizing to non-smooth $\sigma$:}

The idea is to use convolutions to approximate $\sigma \in \CalM$ by smooth functions.
More precisely, using that $\sigma \in \CalM$, one can show for each $\varphi \in C_c^\infty(\CC)$
that the convolution $\varphi \ast \sigma$ can be approximated arbitrarily well
(locally uniformly) by sums of translations of $\sigma$, and thus by shallow
neural networks that use $\sigma$ as the activation function;
see (the proof of) \Cref{lem:ConvolutionApproximation} for the details.

As a consequence, it suffices to show that for a suitable choice of $\varphi$,
one can approximate every continuous function $f$ arbitrarily well
using shallow networks that use $\varphi \ast \sigma$ as the activation function.
Actually, we show something formally weaker, namely that every continuous function $f$
can be approximated arbitrarily well by networks where each neuron has an
activation function $\varphi \ast \sigma$, where $\varphi$ can be chosen independently
for each neuron.
By the same argument as above, this suffices to prove the actual claim.

To prove the universality of networks with activation function $\varphi \ast \sigma$
varying for each neuron, we show that since $\sigma$ is not almost polyharmonic,
there is for each $m \in \N$ a function $\varphi_m \in C_c^\infty (\CC)$ satisfying
$\Delta^m (\varphi_m \ast \sigma) \not\equiv 0$,
and thus also $\Delta^k (\varphi_m \ast \sigma) \not\equiv 0$ for $k \leq m$.
The proof of this is based on noting that $\varphi_\eps \ast \sigma \to \sigma$ in $L_{\loc}^1$,
for a suitable choice of $\varphi_\eps \in C_c^\infty (\CC)$, and on noting that
the property of being (almost) polyharmonic of order $m$ is preserved under this type
of convergence.
This last property is essentially a consequence of \emph{Weyl's lemma} \cite{StroockWeylLemma};
see \Cref{lem:WeaklyPolyharmonicIsPolyharmonic} for the details.
As in \Cref{eq:IntroductionMonomialExtraction}, it follows from
$\Delta^k (\varphi_m \ast \sigma) \not\equiv 0$ for $k \leq m$
that shallow networks using the activation function $\varphi_m \ast \sigma$
can approximate the monomials $z^k \overbar{z}^\ell$ with $0 \leq k,\ell \leq m$ arbitrarily well.
Thus, taking together all the different activation functions $\varphi_m \ast \sigma$ for $m \in \N$,
one recovers all possible monomials and can complete the proof as before.

We remark that this proof device of considering several different activation functions
at once is responsible for many of the technical complications in the proof.

\paragraph{Proving necessity:}

The main observation here is that the operations involved in constructing
a neural network from the activation function $\sigma$ preserve the polyharmonicity.

Precisely, let us assume that $\sigma \in C^\infty(\CC; \CC)$ is polyharmonic,
meaning $\Delta^m \sigma \equiv 0$ for some $m \in \N$;
the generalization to ``almost polyharmonic'' functions is then mainly technical%
\footnote{Note, however, that there is still something to do here.
For instance, for the case of \emph{deep} networks, a similar statement is false,
as \Cref{exa:PathologicalUniversalSigma} shows.}.
Furthermore, it is enough to consider the case of input dimension $d = 1$.
Indeed, if ${f_0 : \CC \to \CC}$ is continuous and cannot be approximated
by networks in $\NN_\sigma^1$, it is easy to see that
${f : \CC^d \to \CC, (z_1,\dots,z_d) \mapsto f(z_1)}$ cannot be approximated
by networks in $\NN_{\sigma}^d$.

Since the affine-linear maps $z \mapsto b_j + w_j z$ are holomorphic,
they satisfy $\delbar (b_j + w_j \, z) = 0$.
Thanks to the chain rule for the Wirtinger derivatives and by \Cref{eq:WirtingerLaplaceRepresentation},
this implies
\begin{align*}
  \Delta^m \big[ \sigma(b_j + w_j \, z) \big]
  & = 4^m \, \del^m \delbar^m \bigl[\sigma(b_j + w_j \, z)\bigr] \\
  & = 4^m \, |w_j|^{2m} \cdot \bigl(\del^m \delbar^m \sigma\bigr) (b_j + w_j z) \\
  & = |w_j|^{2m} \cdot (\Delta^m \sigma) (b_j + w_j z)
    = 0.
\end{align*}
By linearity, this implies $\Delta^m \Phi \equiv 0$ for all $\Phi \in \NN_{\sigma}^1$.
But similar to the space of holomorphic functions, the space of polyharmonic
functions of a fixed order is a proper subspace of the set of continuous functions
and is closed with respect to locally uniform convergence.
This is essentially a consequence of Weyl's lemma \cite{StroockWeylLemma};
see \Cref{lem:WeaklyPolyharmonicIsPolyharmonic} for the details.

Overall, this shows that $\NN_\sigma^1$ does not have the universal approximation
property if $\sigma$ is (almost) polyharmonic.

\subsection{The case of deep networks}%
\label{sub:DeepProofSketch}

\paragraph{Sufficiency for smooth $\sigma$:}

By assumption, $\sigma$ is neither holomorphic nor antiholomorphic
(i.e., $\overline{\sigma}$ is not holomorphic).
Therefore, $\del \sigma \not\equiv 0$ and $\delbar \sigma \not\equiv 0$.
As in \Cref{eq:IntroductionMonomialExtraction}, this implies that shallow networks with
activation function $\sigma$ can approximate the monomials $z$ and $\overbar{z}$,
and hence also the function $z \mapsto \Re z$ arbitrarily well.

Furthermore, since $\sigma$ is not of the form $\sigma(z) = p(z, \overbar{z})$,
one can show for each $m \in \N_0$ that either $\del^m \sigma \not\equiv 0$
or $\delbar^m \sigma \not\equiv 0$.
Again, this implies as in \Cref{eq:IntroductionMonomialExtraction} that shallow networks
with activation function $\sigma$ can either approximate the monomial $z^m$ or the monomial
$\overbar{z}^m$ arbitrarily well.
Note that for real inputs $z \in \R$, it does not matter which of the two cases occurs.

Since we consider deep networks with at least two layers, it follows by composition
that suitable networks $\Phi \in \NN_{\sigma,L}^1$ can approximate arbitrarily well
the function $z \mapsto q(\Re z)$, where $q \in \CC [X]$ is an arbitrary polynomial.
It is then mainly technical to show that this implies that suitable networks
$\Phi \in \NN_{\sigma,L}^1$ can approximate \emph{every} continuous function $f : \CC \to \CC$;
see Step~4 in the proof of \Cref{thm:DeepNetworksSufficientCriterion}.
This establishes universality in the case of input dimension $d = 1$.
The generalization to higher input dimensions is again mainly technical;
see \Cref{lem:FromOneDimensionToHigherDimensions}.

\paragraph{Generalizing to non-smooth $\sigma$:}

Since this is similar to the case of shallow networks, we omit the details in this proof sketch.

\paragraph{Proving necessity:}

As for the case of shallow networks, it is enough to consider the case of
input dimension $d = 1$.
If $\sigma$ is holomorphic, it is straightforward to see that every function
$\Phi \in \NN_{\sigma,L}^1$ is holomorphic as well, and thus satisfies $\delbar \Phi \equiv 0$
and hence $\Delta \Phi \equiv 0$; see \Cref{eq:WirtingerLaplaceRepresentation}.
Likewise, if $\sigma$ is antiholomorphic (meaning that $\overline{\sigma}$ is holomorphic),
then each $\Phi \in \NN_{\sigma,L}^1$ is holomorphic or antiholomorphic,
depending on whether $L$ is even or odd.
Again, this implies $\Delta \Phi \equiv 0$.
Finally, if $\sigma(z) = p(z,\overbar{z})$ is a polynomial of $z$ and $\overbar{z}$,
one can show that each $\Phi \in \NN_{\sigma,L}^1$ satisfies $\Phi(z) = q_\Phi (z,\overbar{z})$
for a two-variable polynomial $q_\Phi \in \CC[X,Y]$ of degree $\deg q_\Phi \leq K = K(p,L)$.
Based on this, one can show $\Delta^m \Phi \equiv 0$, for a suitable $m = m(p,L)$.
Thus, in any case, $\NN_{\sigma,L}^1$ only consists of polyharmonic functions of a fixed degree.
This implies as for shallow networks that the network class $\NN_{\sigma,L}^1$ is not universal.


\section{Proofs}%
\label{sec:Proofs}

In this section we formally prove
\Cref{thm:ShallowUniversalApproximation,thm:DeepUniversalApproximationIntroduction},
following the ideas outlined in the preceding section.
In \Cref{sub:SufficiencyShallow,sub:NecessityShallow}, we consider the case of shallow networks
and separately prove the sufficiency and necessity in \Cref{thm:ShallowUniversalApproximation}.
\Cref{sub:SufficiencyDeep,sub:NecessityDeep} follow the same pattern, but for deep networks.

\subsection{Sufficiency for shallow networks}%
\label{sub:SufficiencyShallow}

In the following, instead of directly working with sets of neural networks,
we will consider the slightly more general setting of (complex) subspaces
$V \subset \{ f : \CC \to \CC \}$ that are \emph{closed under translations and dilations}.
By this we mean that if $\varphi \in V$, then also $\varphi_{a,b} \in V$ for arbitrary
$a,b \in \CC$, where
\begin{equation}
  \varphi_{a,b} : \quad
  \CC \to \CC, \quad
  z \mapsto \varphi(a \, z + b) .
  \label{eq:TranslationDilation}
\end{equation}
The following proposition gives a sufficient condition for such a space $V$
to be dense in the set of all continuous functions.
Here, we only consider subspaces of the set of smooth functions;
the more general case $V \subset \CalM$ will be considered later.

\begin{proposition}\label{prop:GeneralDensityResultWithSmoothness}
  Let $V \subset C^\infty (\CC; \CC)$ be a (complex) vector space that is
  closed under translations and dilations.
  Furthermore, assume that for arbitrary $m,\ell \in \N_0$ there is $\varphi_{m,\ell} \in V$
  with $\del^m \delbar^\ell \varphi_{m,\ell} \not \equiv 0$.
  Then $\overline{V} = C(\CC; \CC)$.
  Here, $\overline{V}$ denotes the closure of $V$
  with respect to the topology of \emph{locally} uniform convergence,
  as defined in \Cref{eq:ClosureDefinition}.
\end{proposition}

The proof of \Cref{prop:GeneralDensityResultWithSmoothness} (given after the following lemma)
is heavily inspired by the arguments in \cite[Step~3 in Section~6]{PinkusUniversalApproximation}.
The main insight of that argument---adapted to the complex setting---%
is contained in the following lemma.
It essentially shows that $\overline{V}$ contains sufficiently many monomials.

\begin{lemma}\label{lem:MonomialExtraction}
  Let $V \subset C^\infty(\CC; \CC)$ be a (complex) vector space that is
  closed under translations and dilations.
  Given arbitrary $\varphi \in V$, $\theta \in \CC$, and $m,\ell \in \N_0$, the map
  \[
    \CC \to \CC, \quad
    z \mapsto z^m \bar{z}^\ell \cdot (\del^m \delbar^\ell \varphi) (\theta)
  \]
  belongs to $\overline{V}$.
  Here, $\overline{V}$ denotes the closure of $V$
  with respect to the topology of \emph{locally} uniform convergence,
  as defined in \Cref{eq:ClosureDefinition}.
\end{lemma}

\begin{proof}
  \textbf{Step 1:} We show for $\varphi \in C^\infty(\CC; \CC)$, $m,\ell \in \N_0$
  and $z, \theta \in \CC$ that
  \begin{equation}
    \del_w^m \, \delbar_w^\ell \big|_{w=0} \big[ \varphi (w z + \theta) \big]
    = z^m \overbar{z}^\ell \cdot (\del^m \delbar^\ell \varphi) (\theta) .
    \label{eq:MonomialReproduction}
  \end{equation}
  To see this, note that $w \mapsto w \, z + \theta$ is holomorphic, so that
  $\del_w \bigl[\, \overline{w z + \theta} \,\bigr] = \overline{\delbar_w [w z + \theta]} = 0$
  and $\del_w (w \, z + \theta) = z$, as well as
  \(
    \delbar_w \big[ \, \overline{w z + \theta} \, \big]
    = \overline{\del_w [w z + \theta]}
    = \overbar{z} ;
  \)
  see \Cref{eq:WirtingerConjugation}.
  Therefore, the chain rule for the Wirtinger derivatives (see \Cref{eq:ChainRule}) shows that
  \[
    \del_w \big[ \varphi(w z + \theta) \big]
    = (\del \varphi) (w z + \theta) \cdot \del_w [w z + \theta]
      + (\delbar \varphi) (w z + \theta)
        \cdot \del_w \bigl[\, \overline{w z + \theta} \,\bigr]
    = z \cdot (\del \varphi) (w z + \theta) .
  \]
  Based on this, a straightforward induction shows that
  $\del_w^{m} [\varphi(w z + \theta)] = z^m \cdot (\del^m \varphi)(w z + \theta)$;
  indeed, for the induction step one uses that
  \[
    \del_w^{m+1} \bigl[\varphi(w z + \theta)\bigr]
    = \del_w \big[ z^m \cdot (\del^m \varphi) (w z + \theta) \big]
    = z^m \cdot \del_w [(\del^m \varphi) (w z + \theta)]
    = z^{m} \cdot z \cdot (\del^{m+1} \varphi) (w z + \theta) .
  \]
  Using exactly the same arguments, one sees that
  \(
    \delbar_w^\ell \big[ \varphi (w z + \theta) \big]
    = \overbar{z}^\ell \cdot (\delbar^\ell \varphi) (w z + \theta) .
  \)
  In combination, this shows
  \[
    \del_w^m \, \delbar_w^\ell \big[ \varphi (w z + \theta) \big]
    = \del_w^m \big[ \overbar{z}^\ell \cdot (\delbar^\ell \varphi) (w z + \theta) \big]
    = \overbar{z}^\ell \, z^m \cdot (\del^m \, \delbar^\ell \varphi) (w z + \theta) .
  \]
  Evaluating this at $w = 0$, we obtain \Cref{eq:MonomialReproduction}.

  \medskip{}

  \noindent
  \textbf{Step 2:} \emph{(Completing the proof):}
  Let $\varphi \in V \subset C^\infty (\CC; \CC)$ and $\theta \in \CC$ be arbitrary.
  Since $V$ is closed under dilations and translations, and since derivatives can be approximated
  by difference quotients, it follows for arbitrary $k,n \in \N_0$ that
  $\psi_{k,n} \in \overline{V}$, where
  \[
    \psi_{k,n} : \quad
    \CC \to \CC, \quad
    z \mapsto \frac{\partial^k \partial^n}{\partial a^k \partial b^n} \Big|_{a = b = 0} \,\,
              \varphi \big( (a + i b) z + \theta \big) .
  \]
  A fully rigorous proof showing that $\psi_{k,n} \in \overline{V}$
  is given in \Cref{prop:DerivativesInTheSpace}.

  By definition of the Wirtinger derivatives and since all partial derivatives
  commute for smooth functions, we see that if we write $w = a + i b$, then
  the operator $\del_w^m \delbar_w^\ell$ is a finite linear combination (with complex coefficients)
  of the operators $\frac{\partial^k \, \partial^n}{\partial a^k \, \partial b^n}$
  for $k,n \in \N_0$.
  Therefore, we see
  \[
    \Bigl(z \mapsto \del_w^m \delbar_w^\ell \big|_{w=0} \big[ \varphi(w z + \theta) \big] \Bigr)
    \in \linspan_{\CC} \bigl\{ \psi_{k,n} \colon k,n \in \N_0 \bigr\}
    \subset \overline{V} .
  \]
  In combination with \Cref{eq:MonomialReproduction}, this yields the claim of the lemma.
\end{proof}

Based on the lemma above, we can easily prove \Cref{prop:GeneralDensityResultWithSmoothness}.

\begin{proof}[Proof of \Cref{prop:GeneralDensityResultWithSmoothness}]
  By assumption, we can find for arbitrary $m,\ell \in \N_0$
  a point $\theta_{m,\ell} \in \CC$ and a function $\varphi_{m,\ell} \in V$ such that
  $\rho_{m,\ell} := (\del^m \delbar^\ell \varphi_{m,\ell})(\theta_{m,\ell}) \neq 0$.
  Thanks to \Cref{lem:MonomialExtraction}, we thus see
  $\big( z \mapsto \rho_{m,\ell} \cdot z^m \overbar{z}^\ell \big) \in \overline{V}$.
  This implies
  \(
    \overline{V}
    \supset \linspan_{\CC}
            \bigl\{
              (z \mapsto z^m \overline{z}^\ell)
              \colon
              m,\ell \in \N_0
            \bigr\} ,
  \)
  where the right-hand side is an algebra of continuous functions that separates the points of
  $\CC$, contains the constant functions, and is closed under conjugation.
  Thus, the complex version of the Stone-Weierstraß theorem
  (see for instance \cite[Theorem~4.51]{FollandRA}) implies that
  $C(\CC; \CC) \subset \overline{\overline{V}} = \overline{V}$, where we recall that the closure
  is taken with respect to the topology of \emph{locally} uniform convergence.
\end{proof}

The next lemma will be used to generalize \Cref{prop:GeneralDensityResultWithSmoothness}
from subspaces $V \subset C^\infty$ to subspaces $V \subset \CalM$.
The proof of the lemma is closely based on that of
\cite[Step~4 in Section~6]{PinkusUniversalApproximation}.

\begin{lemma}\label{lem:ConvolutionApproximation}
  With the space $\CalM$ as in \Cref{sub:IntroductionResults}, let $\sigma \in \CalM$ be arbitrary.
  Then, for any $\varphi \in C_c^\infty (\CC; \CC)$, we have
  \[
    \varphi \ast \sigma
    \in \overline{\strut \linspan_{\CC} \, \{ T_z \, \sigma \colon z \in \CC \}} ,
  \]
  where the closure is defined as in \Cref{eq:ClosureDefinition} and where
  $T_z \sigma : \CC \to \CC, w \mapsto \sigma(w - z)$ denotes the translation of $\sigma$
  by $z \in \CC$, and the convolution $\varphi \ast \sigma : \CC \to \CC$ is given by
  $\varphi \ast \sigma (z) = \int_{\CC} \varphi (w) \, \sigma(z - w) \, d w$,
  where integration is with respect to the Lebesgue measure on $\CC \cong \R^2$.
\end{lemma}

\begin{proof}
  Define $V := \strut \linspan_{\CC} \, \{ T_z \, \sigma \colon z \in \CC \}$.
  Let $K \subset \CC$ be compact and $\eps > 0$.
  Identifying $\CC \cong \R^2$, choose $A > 0$ so large that $\supp \varphi \subset (-A,A)^2$
  and $K \subset (-A,A)^2$.
  We will construct $g \in V$ satisfying $\| g - \varphi \ast \sigma \|_{L^\infty([-A,A]^2)} \leq 2 \eps$.
  Since every compact set $K \subset \CC \cong \R^2$ is contained in $[-A,A]^2$ for sufficiently
  large $A > 0$, and by definition of the closure in \Cref{eq:ClosureDefinition}, this easily
  implies the result.
  The construction of $g$ will be divided into three steps.

  \medskip{}

  \noindent
  \textbf{Step 1:} \emph{(Constructing a suitable covering of the set of discontinuities of $\sigma$):}
  Choose
  \begin{equation}
    0 < \delta
    \leq \eps / \bigl(1 + 10 \, \| \sigma \|_{L^\infty([-2 A, 2 A]^2)} \, \| \varphi \|_{L^\infty}\bigr)
    .
    \label{eq:ConvolutionApproximationDeltaChoice}
  \end{equation}
  Since $\sigma \in \CalM$, we have $\lambda(\overline{D}) = 0$,
  with ${D := \{ z \in \CC \colon \sigma \text{ not continuous at } z \}}$
  and $\lambda$ denoting the Lebesgue measure.
  It is well-known%
  \footnote{One way to see this is to note by outer regularity that there is an open
  set $U \supset \overline{D}$ satisfying $\lambda(U) < \delta/2$.
  Furthermore, the open set $U$ is a countable union $U = \bigcup_{i=1}^\infty Q_i '$
  of cubes $Q_i '$ with disjoint interiors (see e.g.~\cite[Lemma~2.43]{FollandRA}),
  which ensures that $\sum_{i=1}^\infty \lambda(Q_i ') = \lambda(U) < \delta/2$.
  For each $i$, one can then find an \emph{open} cube $Q_i \supset Q_i '$
  with $\lambda(Q_i) < \lambda(Q_i ') + 2^{-(i+1)} \delta$,
  so that $\sum_{i=1}^\infty \lambda(Q_i) < \delta$
  and $\overline{D} \subset U \subset \bigcup_{i=1}^\infty Q_i$.}
  that this implies $\overline{D} \subset \bigcup_{i=1}^\infty Q_i$
  for suitable open cubes $Q_i \subset \R^2$ with $\sum_{i=1}^\infty \lambda(Q_i) \leq \delta$.
  By compactness, $\overline{D} \cap [-2 A, 2 A]^2 \subset \bigcup_{i=1}^N Q_i =: W$
  for a suitable $N \in \N$.

  Since $\sigma$ is (uniformly) continuous on the compact set
  ${\Omega := [-2 A, 2 A]^2 \setminus W}$, we can find $m \in \N$ such that
  $m > 8 A \sqrt{N / \delta}$ and such that
  \begin{equation}
    \big| \sigma(s) - \sigma(t) \big| \leq \frac{\eps}{1 + \| \varphi \|_{L^1}}
    \quad \text{for all} \quad
    s,t \in \Omega \text{ with } \| s - t \|_\infty \leq \frac{2 A}{m} .
    \label{eq:ConvolutionApproximationUniformContinuity}
  \end{equation}

  \medskip{}

  \noindent
  \textbf{Step 2:} \emph{(Constructing $g$):}
  For $k, \ell \in \FirstN{m}$, set
  ${y_{k,\ell} := - (A, A)^T + \frac{2 A}{m} (k-1, \ell-1)^T \! \in \R^2 \cong \CC}$
  and $\Delta_{k,\ell} := y_{k,\ell} + \frac{2 A}{m} [0,1)^2 \subset \R^2 \cong \CC$,
  as well as $\theta_{k,\ell} := \int_{\Delta_{k,\ell}} \varphi(y) \, d y \in \CC$.
  Note that $[-A,A)^2 = \biguplus_{k,\ell \in \FirstN{m}} \Delta_{k,\ell}$
  and also that $g \in V$ for
  \[
    g : \quad
    \CC \to \CC, \quad
    z \mapsto \sum_{k,\ell \in \FirstN{m}}
                \theta_{k,\ell} \, \sigma(z - y_{k,\ell}) .
  \]

  We want to prove $|(\varphi \ast \sigma)(z) - g(z)| \leq 2 \eps$ for all $z \in [-A,A]^2$.
  To this end, first note because of
  $\supp \varphi \subset (-A,A)^2 \subset \biguplus_{k,\ell \in \FirstN{m}} \Delta_{k,\ell}$ that
  \(
    (\varphi \ast \sigma)(z)
    = \sum_{k,\ell \in \FirstN{m}}
        \int_{\Delta_{k,\ell}}
          \varphi(y) \sigma(z - y)
        \, d y ,
  \)
  and hence
  \begin{equation}
    \big| (\varphi \ast \sigma) (z) - g(z) \big|
    \leq \sum_{k,\ell \in \FirstN{m}}
           \int_{\Delta_{k,\ell}}
             |\varphi(y)| \cdot \big| \sigma(z - y) - \sigma(z - y_{k,\ell}) \big|
           \, d y .
    \label{eq:ConvolutionApproximationBasicEstimate}
  \end{equation}

  \medskip{}

  \noindent
  \textbf{Step 3:} \emph{(Completing the proof):}
  Fix $z \in [-A,A]^2$.
  In order to further estimate the right-hand side of \Cref{eq:ConvolutionApproximationBasicEstimate},
  define the ``good index set'' as
  $I_g := \{ (k,\ell) \in \FirstN{m}^2 \colon (z - \Delta_{k,\ell}) \cap W = \emptyset\}$
  and the ``bad index set'' as $I_b:= \FirstN{m}^2 \setminus I_g$.

  For $(k,\ell) \in I_g$, we have $z - \Delta_{k,\ell} \subset [-2A, 2A]^2 \setminus W = \Omega$.
  Thus, we see for each $y \in \Delta_{k,\ell}$ that ${z - y, z - y_{k,\ell} \in \Omega}$
  with $\| (z - y) - (z - y_{k,\ell}) \|_\infty = \| y_{k,\ell} - y \|_\infty \leq \frac{2 A}{m}$.
  In view of \Cref{eq:ConvolutionApproximationUniformContinuity}, this implies
  $|\sigma(z - y) - \sigma(z - y_{k,\ell})| \leq \eps / (1 + \| \varphi \|_{L^1})$.
  Overall, we thus see
  \begin{equation}
    \begin{split}
      & \sum_{(k,\ell) \in I_g}
          \int_{\Delta_{k,\ell}}
            |\varphi(y)| \cdot |\sigma(z - y) - \sigma(z - y_{k,\ell})|
          \, d y \\
      & \leq \frac{\eps}{1 + \| \varphi \|_{L^1}}
             \cdot \sum_{(k,\ell) \in I_g}
                     \int_{\Delta_{k,\ell}}
                       |\varphi(y)|
                     \, d y
        \leq \eps \cdot \frac{\| \varphi \|_{L^1}}{1 + \| \varphi \|_{L^1}}
        \leq \eps .
    \end{split}
    \label{eq:ConvolutionApproximationGoodPart}
  \end{equation}

  On the other hand, for $(k,\ell) \in I_b$ we have $\Delta_{k,\ell} \cap (z - W) \neq \emptyset$.
  Recall that $W = \bigcup_{j=1}^N Q_j$ for certain open cubes $Q_j$,
  say $Q_j = \gamma_j + (-r_j, r_j)^2$,
  where $\sum_{j=1}^N (2 r_j)^2 = \sum_{j=1}^N \lambda(Q_j) \leq \delta$.
  Since $\Delta_{k,\ell}$ is a cube with side-length $2A / m$ satisfying
  $\Delta_{k,\ell} \cap (z - W) \neq \emptyset$, we see
  \[
    \Delta_{k,\ell}
    \subset z - W + \frac{2 A}{m} \, \bigl[-1, 1\bigr]^2
    \subset \bigcup_{j=1}^N
             \bigg(
               z - \gamma_j + \Bigl( \frac{2 A}{m} + r_j \Bigr) [-1,1]^2
             \bigg) ,
  \]
  where the right-hand side is \emph{independent} of the choice of $(k,\ell) \in I_b$.
  Combining this with the elementary estimate $(a + b)^2 \leq 2 \, (a^2 + b^2)$, we thus see
  \begin{align*}
    \sum_{(k,\ell) \in I_b} \!\!\!
      \lambda(\Delta_{k,\ell})
    & = \lambda \bigg(
                  \biguplus_{(k,\ell) \in I_b} \!\!
                    \Delta_{k,\ell}
                \bigg)
      \leq \sum_{j=1}^N \!
             \Big[
               4 \cdot \Big( \frac{2 A}{m} \,+\, r_j \Big)^2
             \Big] \\
    & \leq 8 \sum_{j=1}^N \!
               \Big(
                 \frac{4 A^2}{m^2} \,+\, r_j^2
               \Big)
      \leq \Big( \frac{8 A}{m} \Big)^2 \, N + 2 \delta
      \leq 3 \delta .
  \end{align*}
  Here, the last step used that we chose $m > 8A \sqrt{N / \delta}$ in Step~1.
  The preceding bound allows us to estimate the sum over the ``bad'' indices
  in \Cref{eq:ConvolutionApproximationBasicEstimate} as follows:
  \begin{equation}
    \begin{split}
      \sum_{(k,\ell) \in I_b}
        \int_{\Delta_{k,\ell}}
          |\varphi(y)| \cdot \big| \sigma(z \!-\! y) - \sigma(z \!-\! y_{k,\ell}) \big|
        \, d y
      & \leq 2 \,
             \| \sigma \|_{L^\infty([-2A, 2A]^2)} \,
             \| \varphi \|_{L^\infty}
             \sum_{(k,\ell) \in I_b} \!\!
               \lambda(\Delta_{k,\ell}) \\
      & \leq 6 \delta \,
             \| \sigma \|_{L^\infty([-2A, 2A]^2)} \,
             \| \varphi \|_{L^\infty}
        \leq \eps ,
    \end{split}
    \label{eq:ConvolutionApproximationBadPart}
  \end{equation}
  by our choice of $\delta$ in \Cref{eq:ConvolutionApproximationDeltaChoice}.
  Overall, combining
  \Cref{eq:ConvolutionApproximationBasicEstimate,eq:ConvolutionApproximationGoodPart,eq:ConvolutionApproximationBadPart}
  we see ${|(\varphi \ast \sigma)(z) - g(z)| \leq 2 \eps}$ for all $z \in [-A,A]^2 \supset K$.
  Since $g \in V$, this easily yields the claim $\varphi \ast \sigma \in \overline{V}$.
\end{proof}

Using the preceding lemma, we can now generalize \Cref{prop:GeneralDensityResultWithSmoothness}
to the setting of spaces that may contain non-smooth functions.

\begin{theorem}\label{thm:GeneralDensityResult}
  For $m \in \N_0$, define
  \[
    \CalH_m
    := \big\{
         \sigma : \CC \to \CC
         \quad \colon \quad
         \exists \, \gamma \in C^\infty (\CC; \CC) :
           \sigma = \gamma \text{ almost everywhere and } \Delta^m \gamma \equiv 0
       \big\} .
  \]

  Let $V \subset \CalM$ be a (complex) vector space that is closed under dilations and translations.
  If ${V \nsubseteq \bigcup_{m=0}^\infty \CalH_m}$, then $C(\CC; \CC) \subset \overline{V}$.
\end{theorem}

\begin{proof}
  Define
  \(
    W := \linspan_{\CC}
         \big\{
           \varphi \ast \sigma
           \colon
           \varphi \in C_c^\infty (\CC; \CC) \text{ and } \sigma \in V
         \big\} .
  \)
  Since each $\sigma \in V \subset \CalM$ is locally bounded, standard properties of the
  convolution show $W \subset C^\infty (\CC; \CC)$.
  The idea of the proof is to show that \Cref{prop:GeneralDensityResultWithSmoothness}
  is applicable to $W$---so that $C(\CC; \CC) \subset \overline{W}$---and then
  use \Cref{lem:ConvolutionApproximation} to conclude $\overline{W} \subset \overline{V}$.

  \medskip{}

  \noindent
  \textbf{Step 1:} \emph{(Showing that $W$ is closed under dilations and translations):}
  Clearly, it is enough to show for $a,b \in \CC$, $\varphi \in C_c^\infty (\CC; \CC)$,
  and $\sigma \in V$ that $\big( z \mapsto (\varphi \ast \sigma) (a \, z + b) \big) \in W$.

  Let us first consider the case $a \neq 0$.
  Note that if we identify $\CC \cong \R^2$, then the function $z \mapsto a \, z$
  corresponds to the map
  \(
    \R^2 \to \R^2,
    (x,y) \mapsto \left(
                    \begin{smallmatrix}
                      \Re a & - \Im a \\
                      \Im a & \Re a
                    \end{smallmatrix}
                  \right)
                  \left(\begin{smallmatrix} x \\ y \end{smallmatrix}\right)
    ,
    \vphantom{\sum_j}
  \)
  as follows from the identity
  \(
    a \cdot (x + i \, y)
    = (x \cdot \Re a - y \cdot \Im a)
      + i ( x \cdot \Im a + y \cdot \Re a ) .
    \strut
  \)
  Since
  \({
    \det
    \left(
      \begin{smallmatrix}
        \Re a & - \Im a \\
        \Im a & \Re a
      \end{smallmatrix}
    \right)
    = |a|^2 ,
  }\)
  this justifies the application of the change-of-variables formula in the following calculation,
  in which we use $\sigma_{a,b}$ as defined in \Cref{eq:TranslationDilation}
  and $\varphi_a (z) := \varphi (a z)$:
  \begin{equation}
    \begin{split}
      (\varphi \ast \sigma) (a \, z + b)
      & = \int_{\CC}
            \varphi(w) \,
            \sigma (a \, z + b - w)
          \, d w
        = \int_{\CC}
            \varphi_a (a^{-1} w) \,
            \sigma_{a,b} (z - a^{-1} w)
          \, d w \\
      & = |a|^2 \, \int_{\CC}
                     \varphi_a (v) \cdot \sigma_{a,b} (z - v)
                   \, d v
        = |a|^2 \cdot (\varphi_a \ast \sigma_{a,b}) (z) .
    \end{split}
    \label{eq:ConvolutionTransformation}
  \end{equation}
  Since $\varphi_a \in C_c^\infty(\CC; \CC)$ and $\sigma_{a,b} \in V$,
  this shows $\big( z \mapsto (\varphi \ast \sigma)(a \, z + b) \big) \in W$, as claimed.

  Finally, let us consider the case $a = 0$.
  Note that $z \mapsto (\varphi \ast \sigma)(0 \cdot z + b) = (\varphi \ast \sigma)(b)$
  is a constant function.
  Furthermore, since by assumption $V \nsubseteq \bigcup_{m=0}^\infty \CalH_m$,
  we see in particular that $V \neq \{ 0 \}$.
  Hence, there are $\sigma_0 \in V$ and $\theta_0 \in \CC$ such that $\sigma_0 (\theta_0) \neq 0$.
  Since $V$ is a vector space and closed under dilations and translations, this implies
  $\bigl(z \mapsto \sigma_0 (0 \cdot z + \theta_0)\bigr) \in V$, so that $V$ contains
  the constant function $\Indicator_{\CC} : \CC \to \CC, z \mapsto 1$.
  Now, choose $\varphi_0 \in C_c^\infty (\CC; \CC)$ with $\int_{\CC} \varphi_0(z) \, d z = 1$
  and note $\Indicator_{\CC} = \varphi_0 \ast \Indicator_{\CC} \in W$,
  so that also $W$ contains the constant functions.
  This easily yields the claim in case of $a = 0$.

  \medskip{}

  \noindent
  \textbf{Step 2:} \emph{(Showing that \Cref{prop:GeneralDensityResultWithSmoothness} applies to $W$):}
  Assume towards a contradiction that this is not true; in view of Step~1 this means that there
  exist $m,\ell \in \N_0$ such that $\del^m \delbar^\ell \psi \equiv 0$ for all $\psi \in W$.
  Let $k := \max \{ m, \ell \}$ and note thanks to \Cref{eq:WirtingerLaplaceRepresentation} that
  \begin{equation}
    \Delta^k \psi
    = 4^k \cdot \del^k \delbar^k \psi
    = 4^k \cdot \del^{k-m} \delbar^{k - \ell} \bigl[\, \del^m \delbar^\ell \psi \,\bigr]
    \equiv 0
    \qquad \forall \, \psi \in W .
    \label{eq:GeneralDensityContradictionProperty}
  \end{equation}
  We will show that this implies $V \subset \CalH_k$, which will provide the desired contradiction.

  Define $\eta : \CC \to \R$ via $\eta(z) := C \cdot \exp(1 / (|z|^2 - 1))$ if $|z| < 1$
  and $\eta(z) := 0$ otherwise, where $C > 0$ is chosen such that $\int_{\CC} \eta(z) \, d z = 1$.
  Then it is shown in \cite[Section~C.5]{EvansPDE} that $\eta \in C_c^\infty (\CC; \R)$.
  For $\eps > 0$ define $\eta_\eps : \CC \to \R, z \mapsto \eps^{-2} \cdot \eta(z/\eps)$.
  Let $\sigma \in V \subset \CalM$ be arbitrary and set $\sigma_\eps := \eta_\eps \ast \sigma \in W$.
  By \Cref{eq:GeneralDensityContradictionProperty}, we thus have $\Delta^k \sigma_\eps \equiv 0$.

  Note that $\sigma$ is locally bounded and hence $\sigma \in L_{\loc}^1 (\CC ; \CC)$,
  so that \cite[Theorem~7 in Appendix~C]{EvansPDE} shows $\sigma_\eps \to \sigma$
  as $\eps \downarrow 0$, with convergence in $L^{1}_{\loc}(\CC; \CC)$.
  Using partial integration (which is justified since $\sigma_\eps \in C^\infty(\CC; \CC)$),
  we thus see for arbitrary $\varphi \in C_c^\infty(\CC; \R)$ that
  \[
    \int_{\CC}
      \sigma(z) \, \Delta^k \varphi(z)
    \, d z
    = \lim_{\eps \downarrow 0}
      \int_{\CC}
        \sigma_\eps (z) \, \Delta^k \varphi(z)
      \, d z
    = \lim_{\eps \downarrow 0}
      \int_{\CC}
        \Delta^k \sigma_\eps (z) \, \varphi(z)
      \, d z
    = 0 .
  \]
  This means that $\Delta^k \sigma = 0$ in the sense of distributions.
  It is a folklore fact that this implies ${\sigma = g}$ almost everywhere for a smooth function
  $g \in C^\infty(\CC; \CC)$ with $\Delta^k g \equiv 0$; see \Cref{lem:WeaklyPolyharmonicIsPolyharmonic}
  for a formal proof.
  Hence, $\sigma \in \CalH_k$.
  Since $\sigma \in V$ was arbitrary, this shows $V \subset \CalH_k$,
  contradicting the assumptions of the theorem.

  We have thus shown that \Cref{prop:GeneralDensityResultWithSmoothness} applies to $W$,
  meaning $C(\CC; \CC) \subset \overline{W}$.

  \medskip{}

  \noindent
  \textbf{Step 3:} \emph{(Completing the proof):}
  For $\varphi \in C_c^\infty(\CC; \CC) \vphantom{\sum_j}$ and $\sigma \in V \subset \CalM$,
  \Cref{lem:ConvolutionApproximation} shows that
  \(
    \varphi \ast \sigma
    \in \overline{
          \linspan_{\CC}
          \{
            T_z \, \sigma
            \colon
            z \in \CC
          \}
        }
    \subset \overline{V} ,
  \)
  since $V$ is a vector space and closed under translations.
  Because $\overline{V}$ is a vector space, this implies $W \subset \overline{V}$ and thus
  $C(\CC; \CC) \subset \overline{W} \subset \overline{\overline{V}} = \overline{V}$.
\end{proof}

In order to deduce the ``sufficient part'' of the universal approximation theorem
from \Cref{thm:GeneralDensityResult}, the following lemma will be helpful.

\begin{lemma}\label{lem:FromOneDimensionToHigherDimensions}
  For $f : \CC \to \CC$, $b \in \CC$, and $a \in \CC^d$, define
  \begin{equation}
    f^{(a,b)} : \quad
    \CC^d \to \CC, \quad
    z \mapsto f(b + a^T z) .
    \label{eq:DimensionIncreasingDilation}
  \end{equation}
  Define $\varrho_\CC : \CC \to \CC, z \mapsto \max \{ 0, \Re z \}$
  and let $\CalF \subset \{ f : \CC \to \CC \}$ with $\varrho_\CC \in \overline{\CalF}$.

  If $\CalG \subset \{ g : \CC^d \to \CC \}$ is a (complex) vector space
  that satisfies $\strut f^{(a,b)} \in \CalG$ for all $f \in \CalF$, $b \in \CC$ and $a \in \CC^d$,
  then $C(\CC^d; \CC) \subset \overline{\CalG}$.
\end{lemma}

\begin{proof}
  It is enough to show $C(\CC^d; \R) \subset \overline{\CalG}$;
  the case of complex-valued functions can be handled by noting
  that $\Phi + i \Psi \in \CalG$ if $\Phi, \Psi \in \CalG$.
  Thus, let $\psi \in C(\CC^d; \R)$, $R > 0$ and $\eps > 0$.
  We want to construct $g \in \CalG$ satisfying
  $\| \psi - g \|_{L^\infty(\overline{B_R}(0))} \leq \eps$.

  Define $\underline{\psi} : \R^{2 d} \cong \R^d \times \R^d \to \R, (x,y) \mapsto \psi(x + i y)$
  and $\varrho : \R \to \R, x \mapsto \max \{ 0, x \}$.
  Furthermore, for $z = x + i y$ with $x,y \in \R^d$,
  let us write $\underline{z} := (x, y)^T \in \R^{2 d}$,
  noting that $| \underline{z} | = |z|$.
  Since $\varrho$ is continuous but does not coincide almost everywhere with a polynomial,
  the classical universal approximation theorem (see \cite{PinkusUniversalApproximation})
  yields $M \in \N$ and $\beta_1,\dots,\beta_M \in \R^{2 d}$
  as well as $\alpha_1,\dots,\alpha_M, \gamma_1,\dots,\gamma_M \in \R$
  such that $\big\| \underline{\psi} - h \big\|_{L^\infty(\overline{B_R}(0))} \leq \frac{\eps}{2}$,
  where $h : \R^{2 d} \to \R$ is defined by
  $h(v) = \sum_{j=1}^M \alpha_j \, \varrho(\gamma_j + \beta_j^T v)$.

  Let us write $\R^{2 d} \ni \beta_j = \bigl(\beta_j^{(1)}, \beta_j^{(2)}\bigr) \in \R^d \times \R^d$
  and define $w_j := \beta_j^{(1)} - i \, \beta_j^{(2)} \in \CC^d$ and
  $b_j := \gamma_j \in \R \subset \CC$ for ${j \in \{ 1, \dots, M \}}$.
  A straightforward calculation shows
  \begin{equation}
    \Re \bigl(b_j + w_j^T z\bigr)
    = \gamma_j + \beta_j^T \underline{z}
    \qquad \text{and hence} \qquad
    \varrho_{\CC} \bigl(b_j + w_j^T z\bigr)
    = \varrho \bigl(\gamma_j + \beta_j^T \underline{z}\bigr)
    \label{eq:RealLinearformAsComplex}
  \end{equation}
  for all $j \in \FirstN{M}$ and $z \in \CC^d$.

  \smallskip{}

  Set $A := 1 + \sum_{j=1}^M |\alpha_j|$ and choose $C > 0$ with $\big| b_j + w_j^T z \big| \leq C$
  for all $z \in \overline{B_R}(0)$ and $j \in \FirstN{M}$.
  Since ${\varrho_{\CC} \in \overline{\CalF}}$, there is $\phi \in \CalF$ satisfying
  $\| \varrho_{\CC} - \phi \|_{L^\infty(\overline{B_C}(0))} \leq \frac{\eps}{4 A}$.
  With notation as in \Cref{eq:DimensionIncreasingDilation}, define
  \[
    g := \sum_{j=1}^M
           \alpha_j \, \phi^{(w_j,b_j)}
    .
  \]
  Since $\CalG$ is a vector space with $\phi^{(w,b)} \in \CalG$
  for all $b \in \CC$ and $w \in \CC^d$, we see $g \in \CalG$.
  Finally, recall that ${b_j + w_j^T z \in \overline{B_C}(0) \subset \CC}$
  for all $z \in \overline{B_R}(0) \subset \CC^d$.
  Thus, we see by definition of $h$ and in view of \Cref{eq:RealLinearformAsComplex} that
  \begin{align*}
    |\psi(z) - g(z)|
    & \leq \big| \underline{\psi}(\underline{z}) - h(\underline{z}) \big|
         + \big| h(\underline{z}) - g(z) \big| \\
    & \leq \frac{\eps}{2}
           + \sum_{j=1}^M
               |\alpha_j|
               \cdot \big|
                       \varrho_{\CC}\bigl(b_j + w_j^T z\bigr)
                       - \phi\bigl(b_j + w_j^T z\bigr)
                     \big|
      \leq \frac{\eps}{2}
           + \frac{\eps}{4 A} \sum_{j=1}^M |\alpha_j|
      \leq \eps
  \end{align*}
  for all $z \in \overline{B_R}(0)$, as desired.
\end{proof}

We close this subsection by deducing the ``sufficient part'' of
\Cref{thm:ShallowUniversalApproximation} from \Cref{thm:GeneralDensityResult}
and \Cref{lem:FromOneDimensionToHigherDimensions}.

\begin{theorem}\label{thm:ShallowUniversalApproximationSufficient}
  Let $\sigma \in \CalM$ and suppose that there \emph{does not} exist a (complex)
  polyharmonic function $h \in C^\infty(\CC; \CC)$ satisfying $\sigma = h$ almost everywhere.
  Then for each input dimension $d \in \N$, the set $\NN_\sigma^d$ of shallow
  complex-valued neural networks with activation function $\sigma$ and $d$-dimensional input
  has the universal approximation property,
  meaning $C(\CC^d; \CC) \subset \overline{\NN_\sigma^d}$.
\end{theorem}

\begin{proof}
  For $a,b \in \CC$, let $\sigma_{a,b}$ be as defined in \Cref{eq:TranslationDilation}.
  Define $V := \linspan_{\CC} \{ \sigma_{a,b} \colon a,b \in \CC  \}$,
  noting that $V \subset \CalM$ and that $\sigma = \sigma_{1,0} \in V$,
  so that $V \nsubseteq \bigcup_{m=0}^\infty \CalH_m$ by our assumption on $\sigma$.
  Moreover, using the identity $(\sigma_{a,b})_{\alpha,\beta} = \sigma_{a \alpha, b + a \beta}$
  it follows that $V$ is invariant under dilations and translations.
  Therefore, we can apply \Cref{thm:GeneralDensityResult} to conclude
  that $C(\CC;\CC) \subset \overline{V}$.

  Define $\CalF := V$ and note that $\CalG := \NN_\sigma^d \subset \{ g : \CC^d \to \CC \}$
  is a vector space.
  Furthermore, note for $\alpha,\beta,b \in \CC$ and $a \in \CC^d$ with notation as in
  \Cref{lem:FromOneDimensionToHigherDimensions} for arbitrary $z \in \CC^d$ that
  \[
    (\sigma_{\alpha,\beta})^{(a,b)} (z)
    = \sigma_{\alpha,\beta} (b + a^T z)
    = \sigma\bigl(\beta + \alpha (b + a^T z)\bigr)
    = \sigma \bigl(\beta + \alpha b + (\alpha a)^T z\bigr) ,
  \]
  where the function on the right-hand side is an element of $\CalG$.
  This easily implies $f^{(a,b)} \in \CalG$ for arbitrary $f \in \CalF = V$.
  Therefore, \Cref{lem:FromOneDimensionToHigherDimensions} shows
  $C(\CC^d;\CC) \subset \overline{\CalG} = \overline{\NN_\sigma^d}$, as claimed.
\end{proof}

\subsection{Necessity for shallow networks}%
\label{sub:NecessityShallow}

In this section, we prove the ``necessary part'' of \Cref{thm:ShallowUniversalApproximation}.
Precisely, we will prove the following:

\begin{theorem}\label{thm:ShallowUniversalApproximationNecessary}
  Let $\sigma : \CC \to \CC$ and suppose that $\sigma = g$ almost everywhere for a
  (complex-valued) polyharmonic function $g \in C^\infty(\CC; \CC)$.
  Then for any $d \in \N$ and with $\NN_\sigma^d$ as in
  \Cref{sub:IntroductionResults}, we have $C(\CC^d; \CC) \nsubseteq \overline{\NN_\sigma^d} \strut$.
  In fact, there exist $f \in C_c (\CC^d; \R)$ and $\eps > 0$ such that
  $\| f - \Phi \|_{L^1(B_1(0))} \geq \eps$ for all $\Phi \in \NN_\sigma^d$.
\end{theorem}

The main ingredient for the proof is the following lemma
which essentially shows that a space consisting of polyharmonic functions
\emph{of a fixed order} cannot have the universal approximation property,
even if one relaxes the topology of locally uniform convergence to the topology
of local $L^1$ convergence.

\begin{lemma}\label{lem:PolyharmonicityObstructsUniversality}
  Let $d \in \N$ and $\CalF \subset \{ F : \CC^d \to \CC \colon F \text{ measurable} \}$,
  and assume that there exists $m \in \N_0$ with the following property:
  For each $F \in \CalF$ and each $w \in \CC^{d-1}$,
  there exists $g = g_{w,F} \in C^{\infty}(\CC;\CC)$ with $\Delta^m g \equiv 0$
  and such that $F(z,w) = g(z)$ for almost all $z \in \CC$.

  Then there exist $\eps > 0$ and $f \in C_c(\CC^d; \R)$ such that
  $\| f - F \|_{L^1(B_1(0))} \geq \eps$ for all $F \in \CalF$.
  In particular this implies $C(\CC^d; \CC) \nsubseteq \overline{\mathcal{F}}$.
\end{lemma}

\begin{rem*}
  In case of $d = 1$, the assumption is to be understood as stating that for each $F \in \CalF$
  there is $g \in C^\infty(\CC;\CC)$ with $\Delta^m g \equiv 0$ and $F = g$ almost everywhere.
\end{rem*}

\begin{proof}
  We only provide the proof for the case $d \geq 2$ and leave the modification for the (easier)
  case $d = 1$ to the reader.
  For brevity, given $g : \CC^d \to \CC$, define
  $g^{[w]} : \CC \to \CC, z \mapsto g(z,w)$ for $w \in \CC^{d-1}$.

  Let $U_1 := B_{1/2}(0) \subset \CC$ and $U_2 := B_{1/2}(0) \subset \CC^{d-1}$,
  noting that $U_1 \times U_2 \subset B_1(0) \subset \CC^d$.
  Choose $\gamma \in C_c(\CC^d; \R)$ with $\gamma \equiv 1$ on $\overline{B_1}(0) \subset \CC^d$
  and define
  \[
    f : \quad
    \CC^d \to \R, \quad
    z = (z_1,\dots,z_d) \mapsto \gamma(z) \cdot \max \bigl\{ 0, \Re z_1 \bigr\} .
  \]
  Assume towards a contradiction that the claim of the lemma is false;
  this implies existence of a sequence $(F_n)_{n \in \N} \subset \CalF$ such that if we set
  \[
    H_n : \quad
    U_2 \to [0,\infty], \quad
    w \mapsto \int_{U_1}
                \big| f(z,w) - F_n(z,w) \big|
              \, d z ,
  \]
  then
  \[
    \| H_n \|_{L^1(U_2)}
    =    \int_{U_1 \times U_2}
           \big| f(z,w) - F_n(z,w) \big|
         \, d (z,w)
    \leq \| f - F_n \|_{L^1(B_1(0))}
    \xrightarrow[n\to\infty]{} 0.
  \]
  It is well-known (see for instance \cite[Corollary~2.32]{FollandRA}) that this implies existence
  of a subsequence $(H_{n_k})_{k \in \N}$ satisfying $H_{n_k} (w) \to 0$ as $k \to \infty$,
  for almost all $w \in U_2$.
  Since $U_2$ has positive measure, we can choose $w_0 \in U_2$ satisfying
  $\big\| f^{[w_0]} - F_{n_k}^{[w_0]} \big\|_{L^1(U_1)} = H_{n_k} (w_0) \to 0$ as $k \to \infty$.
  Furthermore, by the assumptions of the lemma there is for each $k \in \N$ some
  $g_k \in C^\infty(\CC;\CC)$ satisfying $\Delta^m g_k \equiv 0$ and $F_{n_k}^{[w_0]} = g_k$
  almost everywhere.

  For arbitrary $\varphi \in C_c^\infty (U_1; \R)$, this implies
  \begin{align*}
    \int_{U_1}
      f^{[w_0]} (z) \cdot \Delta^m \varphi (z)
    \, d z
    & = \lim_{k \to \infty}
        \int_{U_1}
          F_{n_k}^{[w_0]} (z) \, \Delta^m \varphi(z)
        \, d z
      = \lim_{k \to \infty}
        \int_{U_1}
          g_k (z) \, \Delta^m \varphi(z)
        \, d z \\
    & = \lim_{k \to \infty}
        \int_{U_1}
          \varphi(z) \, \Delta^m g_k (z)
        \, d z
      = 0 ,
  \end{align*}
  meaning that $f^{[w_0]} \in L_{\loc}^1(U_1; \R)$ is weakly polyharmonic of order $m$.
  As shown in \Cref{lem:WeaklyPolyharmonicIsPolyharmonic}, this yields existence of
  $h \in C^\infty(U_1;\R)$ with $f^{[w_0]} = h$ almost everywhere on $U_1$.
  But for $z \in U_1$ we have $(z, w_0) \!\in\! B_1(0)$ and hence
  ${f^{[w_0]}(z) \!=\! f(z, w_0) \!=\! \gamma(z,w_0) \cdot \max \{ 0, \Re z \} \!=\! \max \{ 0, \Re z \}}$.
  On the one hand, this shows that $h$ and $f^{[w_0]}$ are both continuous on $U_1$,
  so that $h = f^{[w_0]}$ everywhere on $U_1$, not just almost anywhere.
  On the other hand, we thus get for all $z \in U_1$ that $h(z) = f^{[w_0]}(z) = \max \{ 0, \Re z \}$,
  where the right-hand side is not a smooth function of $z \in U_1 = B_{1/2}(0) \subset \CC$.
  This contradicts the smoothness of $h$.

  Finally, if we had $C(\CC^d; \CC) \subset \overline{\CalF}$, there would be a sequence
  $(F_n)_{n \in \N} \subset \CalF$ satisfying $F_n \to f$ uniformly on $\overline{B_1}(0)$,
  and hence $\| f - F_n \|_{L^1(B_1(0))} \to 0$ as well, in contradiction to what we just showed.
\end{proof}

We close this subsection by using the above lemma to prove
\Cref{thm:ShallowUniversalApproximationNecessary}.

\begin{proof}[Proof of \Cref{thm:ShallowUniversalApproximationNecessary}]
  Since $g$ is polyharmonic, there is $m \in \N$ with $\Delta^m g \equiv 0$.
  The remainder of the proof is divided into three steps.

  \medskip{}

  \noindent
  \textbf{Step 1:} \emph{(With $g_{a,b}$ as in \Cref{eq:TranslationDilation},
  we have $\Delta^m g_{a,b} \equiv 0$ for all $a,b \in \CC$):}
  To see this, note that $z \mapsto b + a \, z$ is holomorphic, so that elementary
  properties of the Wirtinger derivatives (see \Cref{sec:WirtingerCalculus}) show
  $\overline{\delbar \, \overline{(b + a \, z)}} = \del (b + a z) = a$
  and $\del \overline{(b + a z)} = \overline{\delbar (b + a \, z)} = 0$.
  Thus, the chain rule for Wirtinger derivatives (see \Cref{eq:ChainRule})
  shows for $h \in C^\infty(\CC;\CC)$ that
  \[
    \del \bigl[h(b + a \, z)\bigr]
    = (\del h) (b + a \, z) \cdot \del(b + a \, z)
      + (\delbar h)(b + a \, z) \cdot \del \overline{(b + a \, z)}
    = a \cdot (\del h) (b + a \, z)
  \]
  and
  \[
    \delbar \big[ h(b + a \, z) \big]
    = (\del h)(b + a \, z) \cdot \delbar (b + a \, z)
      + (\delbar h) (b + a \, z) \cdot \delbar \overline{b + a \, z}
    = \overline{a} \cdot (\delbar h)(b + \, a z) .
  \]
  Based on these identities and using \Cref{eq:WirtingerLaplaceRepresentation},
  an induction shows that
  \[
    \Delta^m g_{a, b} (z)
    = 4^m \cdot \del^m \delbar^m \big[ g (b + a \, z) \big]
    = 4^m \cdot a^m \overline{a}^m \cdot (\del^m \delbar^m g) (b + a \, z)
    = |a|^{2m} \cdot (\Delta^m g) (b + a \, z)
    = 0
  \]
  for all $z \in \CC$, as claimed.

  \medskip{}

  \noindent
  \textbf{Step 2:} \emph{(There are $\eps > 0$ and $f \in C_c(\CC^d; \R)$
  with $\| f - \Psi \|_{L^1(B_1(0))} \geq \eps$ for all $\Psi \in \NN_g^d$):}
  Note for $b \in \CC$, $w \in \CC^{d-1}$ and $a \in \CC^d$ that there exists
  $d = d(a,b,w) \in \CC$ satisfying $b + a^T (z,w) = d + a_1 \, z$ for all $z \in \CC$.
  Therefore, Step~1 shows that
  \[
    \Delta^m \bigl(z \mapsto g(b + a^T (z,w))\bigr)
    = \Delta^m g_{a_1, d} \equiv 0
    .
  \]
  By linearity, by definition of $\NN_{g}^d$, and because of $\Delta^m (1) \equiv 0$
  (since $m \in \N$), we thus see that $\Delta^m \bigl(z \mapsto \Psi (z,w) \bigr) \equiv 0$
  for all $\Psi \in \NN_g^d$ and $w \in \CC^{d-1}$.
  Thus, setting $\CalF := \NN_g^d$, \Cref{lem:PolyharmonicityObstructsUniversality} yields
  $f \in C_c(\CC^d;\R)$ and $\eps > 0$
  satisfying $\| f - F \|_{L_1(B_1(0))} \geq \eps$ for all $F \in \CalF$.
  This is precisely what was claimed in this step.

  \medskip{}

  \noindent
  \textbf{Step 3:} \emph{(For each $\Phi \in \NN_\sigma^d$, there is $\Psi \in \NN_g^d$
  with $\Phi = \Psi$ almost everywhere):}
  Let $N := \{ z \in \CC \colon \sigma(z) \neq g(z) \}$.
  By assumption, $N \subset \CC$ is a (Lebesgue) null-set.
  Each $\Phi \in \NN_\sigma^d$ is of the form
  ${\Phi(z) = c + \sum_{j=1}^k \theta_j \, \sigma(b_j + w_j^T z)}$
  for certain $k \in \N$, ${\theta_1,b_1,\dots,\theta_k,b_k,c \in \CC}$
  and $w_1,\dots,w_k \in \CC^d$.

  Let $J_0 := \bigl\{ j \in \FirstN{k} \colon w_j = 0 \bigr\}$
  and ${J_1 := \FirstN{k} \setminus J_0}$.
  For each $j \in J_1$, \Cref{lem:HyperplaneNullSet} shows that
  $M_j := \{ z \in \CC^d \colon b_j + w_j^T z \in N \}$ is a null-set,
  and hence so is $M := \bigcup_{j \in J_1} M_j$.
  For each $j \in J_0$, we have $\theta_j \, \sigma (b_j + w_j^T z) \equiv \theta_j \, \sigma(b_j)$
  and hence $c + \sum_{j \in J_0} \theta_j \, \sigma(b_j + w_j^T z) \equiv c'$
  for a suitable $c' \in \CC$.
  On the other hand, for $j \in J_1$ and $z \in \CC^d \setminus M$, we have
  $\sigma(b_j + w_j^T z) = g(b_j + w_j^T z)$, which shows that if we define $\Psi : \CC^d \to \CC$
  by $\Psi (z) := c' + \sum_{j \in J_1} \theta_j \, g(b_j + w_j^T z)$
  then $\Psi \in \NN_g^d$ and $\Phi(z) = \Psi(z)$
  for all $z \in \CC^d \setminus M$ and hence almost everywhere.

  \medskip{}

  Finally, let $f \in C_c(\CC^d;\R)$ as provided by Step~2.
  For each $\Phi \in \NN_\sigma^d$, Step~3 yields ${\Psi \in \NN_g^d}$
  with $\Phi = \Psi$ almost everywhere, so that
  ${\| f - \Phi \|_{L^1(B_1(0))} = \| f - \Psi \|_{L^1(B_1(0))} \geq \eps \!>\! 0}$.
  Exactly as in the proof of \Cref{lem:PolyharmonicityObstructsUniversality},
  this also implies $C(\CC^d; \CC) \nsubseteq \overline{\NN_\sigma^d}$.
\end{proof}

Finally, we remark that \Cref{thm:ShallowUniversalApproximation} is an immediate consequence
of combining \Cref{thm:ShallowUniversalApproximationSufficient,thm:ShallowUniversalApproximationNecessary}.

\subsection{Sufficiency for deep networks}%
\label{sub:SufficiencyDeep}

In this section, we consider the case of deep complex-valued networks;
that is, complex-valued networks with more than one hidden layer.
Somewhat surprisingly, these deep networks enjoy the universal approximation property
for a strictly larger class of activation functions than for the case of shallow networks.
Before we state and prove this result rigorously, let us first provide a
precise definition of such deep complex-valued networks.

\begin{definition}\label{def:DeepComplexNetworks}
  Let $d,L \in \N$.
  We define the set $\CalW_L^d$ of all \emph{complex network weights}
  with $d$-dimensional input and $L$ hidden layers as
  \[
    \CalW_L^d
    := \bigg\{
         \big(
           (A_0, b_0),
           \dots,
           (A_L, b_L)
         \big)
         \quad \colon \quad
         \begin{array}{l}
           N_0,\dots,N_{L+1} \in \N,
           \quad
           N_0 = d,
           \quad
           N_{L+1} = 1, \\[0.1cm]
           \text{and }
           A_j \in \CC^{N_{j+1} \times N_j}, \quad
           b_j \in \CC^{N_{j+1}}
         \end{array}
       \bigg\} .
  \]
  Given network weights $\Theta = \big( (A_0,b_0),\dots,(A_L,b_L) \big) \in \CalW_L^d$
  and any function $\sigma : \CC \to \CC$, we define the associated \emph{network function}
  $\CalN_\sigma \Theta : \CC^d \to \CC$ via $\CalN_\sigma \Theta (z) = z^{(L)}$, where
  \[
    z^{(-1)} := z,
    \quad
    z^{(j)} := \sigma\bigl(b_j + A_j \, z^{(j-1)}\bigr) \quad \text{for } j \in \{ 0,\dots,L-1 \},
    \quad \text{and} \quad
    z^{(L)} := b_L + A_L \, z^{(L-1)} .
  \]
  Here, the activation function $\sigma$ is applied componentwise, meaning
  $\sigma(z) = \big( \sigma(z_1),\dots,\sigma(z_N) \big)$ for $z = (z_1,\dots,z_N) \in \CC^N$.

  Finally, we define the set of all network functions with activation function $\sigma$,
  $L$ hidden layers, and $d$-dimensional input as
  \[
    \NN_{\sigma,L}^d
    := \big\{
         \CalN_\sigma \Theta
         \,\, \colon \,\,
         \Theta \in \CalW_L^d
       \big\} .
  \]
\end{definition}

\begin{rem*}
  It is not hard to see that $\NN_{\sigma,1}^d = \NN_\sigma^d$,
  with the set $\NN_\sigma^d$ of shallow complex-valued networks
  as defined in \Cref{sub:IntroductionResults}.
\end{rem*}

Our analysis of the universal approximation property for such deep networks will crucially
depend on the following closure properties of the network classes $\NN_{\sigma,L}^d$.
The (mainly technical) proof is deferred to \Cref{sub:NetworkClosurePropertiesProof}.

\begin{lemma}\label{lem:NetworkClosureProperties}
  Let $\sigma : \CC \to \CC$ and $d, L, T \in \N$.
  Then the following hold:
  \begin{enumerate}[label=\alph*)]
    \item \label{enu:NetworksVectorSpace}
          If $\Phi, \Psi \in \NN_{\sigma,L}^d$ and $\alpha, \beta \in \CC$,
          then $\alpha \Phi + \beta \Psi \in \NN_{\sigma,L}^d$ as well.

    \item \label{enu:NetworksLowDimToHighDim}
          If $\Phi \in \NN_{\sigma,L}^1$ and $a \in \CC^d$, $b \in \CC$, then
          $\Xi : \CC^d \to \CC, z \mapsto \Phi(b + a^T z)$ satisfies
          $\Xi \in \NN_{\sigma,L}^d$.

    \item \label{enu:NetworksHighDimToLowDim}
          If $\Phi \in \NN_{\sigma,L}^d$ and $a,b \in \CC^d$, then
          $\Xi : \CC \to \CC, z \mapsto \Phi(b + z a)$ satisfies $\Xi \in \NN_{\sigma,L}^1$.

    \item \label{enu:NetworksComposition}
          If $\Phi \in \NN_{\sigma,L}^d$ and $\Psi \in \NN_{\sigma,T}^1$,
          then $\Psi \circ \Phi \in \NN_{\sigma, L+T}^d$.
  \end{enumerate}
\end{lemma}

With these properties, we can now prove the ``sufficient part''
of \Cref{thm:DeepUniversalApproximationIntroduction}.
Precisely, we prove the following:

\begin{theorem}\label{thm:DeepNetworksSufficientCriterion}
  Let $\sigma \in \CalM$ and assume that \emph{none} of the following properties hold:
  \begin{enumerate}[label=\alph*)]
    \item we have $\sigma (z) = p(z, \overline{z})$ for almost all $z \in \CC$,
          where $p \in \CC[X,Y]$ is a complex polynomial of two variables,

    \item we have $\sigma = g$ almost everywhere or $\sigma = \overline{g}$ almost everywhere,
          where $g : \CC \to \CC$ is an entire function.
  \end{enumerate}
  Then we have $C(\CC^d; \CC) \subset \overline{\NN_{\sigma,L}^d}$ for arbitrary $d \in \N$
  and $L \in \N_{\geq 2}$.
\end{theorem}

\begin{proof}
  \textbf{Step 1:} \emph{(There are $\varphi_1, \varphi_2 \in C_c^\infty (\CC;\R)$ satisfying
  $\del (\varphi_1 \ast \sigma) \not \equiv 0$ and ${\delbar (\varphi_2 \ast \sigma) \not \equiv 0}$):}
  Assume towards a contradiction that this is false.
  Then there is a differential operator ${\delta \in \{ \del, \delbar \}}$
  such that $\delta (\varphi \ast \sigma) \equiv 0$ for all $\varphi \in C_c^\infty(\CC;\R)$.

  Define $\eta : \CC \to \R$ via $\eta(z) := C \cdot \exp(1 / (|z|^2 - 1))$ if $|z| < 1$
  and $\eta(z) := 0$ otherwise, where $C > 0$ is chosen such that $\int_{\CC} \eta(z) \, d z = 1$.
  Then it is shown in \cite[Section~C.5]{EvansPDE} that $\eta \in C_c^\infty (\CC; \R)$.
  For $\eps > 0$ define $\eta_\eps : \CC \to \R, z \mapsto \eps^{-2} \cdot \eta(z/\eps)$
  and $\sigma_\eps := \eta_\eps \ast \sigma$.
  By assumption, we have $\delta \, \sigma_\eps \equiv 0$.
  Thanks to \Cref{eq:WirtingerLaplaceRepresentation}, this implies in particular that
  $\Delta \sigma_\eps \equiv 0$.

  Note that $\sigma$ is locally bounded and hence $\sigma \in L_{\loc}^1 (\CC ; \CC)$,
  so that \cite[Theorem~7 in Appendix~C]{EvansPDE} shows $\sigma_\eps \to \sigma$
  as $\eps \downarrow 0$, with convergence in $L^{1}_{\loc}(\CC; \CC)$.
  Using partial integration (which is justified since $\sigma_\eps \in C^\infty(\CC; \CC)$),
  we thus see for arbitrary $\varphi \in C_c^\infty(\CC; \R)$ that
  \begin{equation}
    \int_{\CC}
      \sigma (z) \, \Delta \varphi(z)
    \, d z
    = \lim_{\eps \downarrow 0}
      \int_{\CC}
        \sigma_\eps (z) \, \Delta \varphi(z)
      \, d z
    = \lim_{\eps \downarrow 0}
      \int_{\CC}
        \varphi(z) \, \Delta \sigma_\eps (z)
      \, d z
    = 0 \, ;
    \label{eq:DeepNetworksWeakHarmonicity}
  \end{equation}
  that is, $\sigma$ is weakly harmonic.
  Thanks to \Cref{lem:WeaklyPolyharmonicIsPolyharmonic}, this implies $\sigma = g$ almost everywhere
  for a suitable function $g \in C^\infty (\CC; \CC)$.
  Recalling that $\delta \in \{ \del, \delbar \}$ and $\delta \, \sigma_\eps \equiv 0$,
  we see for arbitrary $\varphi \in C_c^\infty (\CC; \R)$ via partial integration that
  \begin{equation}
    \begin{split}
      \int_{\CC}
        \varphi(z) \delta g (z)
      \, d z
      & = - \int_{\CC}
              g (z) \, \delta \, \varphi (z)
            \, d z
        = - \int_{\CC}
              \sigma (z) \, \delta \, \varphi (z)
            \, d z \\
      & = - \lim_{\eps \downarrow 0}
            \int_{\CC}
              \sigma_\eps (z) \, \delta \varphi (z)
            \, d z
        = \lim_{\eps \downarrow 0}
          \int_{\CC}
            \varphi(z) \, \delta \, \sigma_\eps (z)
          \, d z
        = 0 .
    \end{split}
    \label{eq:DeepNetworksWeakDifferentialOperator}
  \end{equation}
  By the fundamental lemma of the calculus of variations (see for instance \cite[Lemma~4.22]{AltFA}),
  this implies $\delta \, g = 0$ almost everywhere, and then everywhere since $\delta \, g$
  is continuous.
  In case of $\delta = \delbar$ this implies that $g$ is holomorphic;
  see \Cref{sec:WirtingerCalculus}.
  Since $\sigma = g$ almost everywhere, $\sigma$ then coincides almost everywhere with a
  holomorphic function, in contradiction to our assumptions.
  If otherwise $\delta = \del$, then $\delbar \overline{g} = \overline{\del g} \equiv 0$,
  so that $h = \overline{g}$ is holomorphic and $\sigma = g = \overline{h}$ almost everywhere,
  again in contradiction to our assumptions.

  \medskip{}

  \noindent
  \textbf{Step 2:} \emph{(For each $m \in \N_0$, there is $\psi_m \in C_c^\infty(\CC; \R)$
  with $\del^m (\psi_m \ast \sigma) \not\equiv 0$ or ${\delbar^m (\psi_m \ast \sigma) \not \equiv 0}$):}
  Assume towards a contradiction that this is false, meaning there is $m \in \N_0$ such that
  for each $\psi \in C_c^\infty(\CC; \R)$ we have $\del^m (\psi \ast \sigma) \equiv 0$
  \emph{and} $\delbar^m (\psi \ast \sigma) \equiv 0$.
  Let $\eta, \eta_\eps$, and $\sigma_\eps$ as defined in Step~1 and note
  $\del^m \sigma_\eps \equiv 0 \equiv \delbar^m \sigma_\eps$.
  Thanks to \Cref{eq:WirtingerLaplaceRepresentation} this implies
  $\Delta^m \sigma_\eps = 4^m \del^m \delbar^m \sigma_\eps \equiv 0$.

  As in \Cref{eq:DeepNetworksWeakHarmonicity}, we thus see
  $\int_{\CC} \sigma(z) \Delta^m \varphi (z) \, d z = 0$ for all $\varphi \in C_c^\infty(\CC;\R)$,
  meaning that $\sigma$ is weakly polyharmonic of order $m$.
  Thanks to \Cref{lem:WeaklyPolyharmonicIsPolyharmonic}, this implies that there exists
  $g \in C^\infty(\CC; \R)$ satisfying $\sigma = g$ almost everywhere.
  Since $\del^m \sigma_\eps \equiv 0 \equiv \delbar^m \sigma_\eps$ for all $\eps > 0$,
  we see using similar arguments as in \Cref{eq:DeepNetworksWeakDifferentialOperator}
  that $\del^m g \equiv 0 \equiv \delbar^m g$.
  The condition $\delbar^m g \equiv 0$ means that $g$ is \emph{polyanalytic of order $m$};
  see \cite[Equation~(1.3)]{BalkPolyanalyticFunctions}.
  As shown in \cite[Bottom of Page 10]{BalkPolyanalyticFunctions}, this implies that we can write
  $g(z) = \sum_{k=0}^{m-1} \overline{z}^k \, a_k (z)$ for all $z \in \CC$,
  with holomorphic functions $a_k : \CC \to \CC$.

  \smallskip{}

  Denoting by $a_k^{(m)}$ the \emph{complex} derivative of $a_k$ of order $m$,
  we claim that $a_k^{(m)} \equiv 0$ for all ${k \in \{ 0,\dots,m-1 \}}$.
  Assume towards a contradiction that this is not so, and let ${K \in \{ 0,\dots,m-1 \}}$
  be maximal with $a_K^{(m)} \not\equiv 0$.
  Since $\delbar f \equiv 0$ and $\del f = f'$ for each holomorphic function $f$
  (with $f'$ denoting the complex derivative of $f$), we have
  $\delbar a_k(z) \equiv 0$ and $\del \bar{z}^k = \overline{\delbar z^k} \equiv 0$
  as well as $\del^m a_k (z) = a_k^{(m)} (z)$ and $\delbar^K \bar{z}^k = \overline{\del^K z^k} = 0$
  for $k < K$, while $\delbar^K \bar{z}^K = \overline{\del^K z^K} = K!$.
  Combining this with the product rule for Wirtinger derivatives (see \Cref{sec:WirtingerCalculus}
  and in particular \Cref{eq:SpecialProductRule}) and recalling the maximality of $K$,
  we see because of $\del^m g \equiv 0$ that
  \[
    0
    = \delbar^K \del^m g(z)
    = \sum_{k=0}^{m-1}
        \big[
          (\delbar^K \bar{z}^k) \cdot a_k^{(m)}(z)
        \big]
    = \sum_{k=0}^{K}
        \big[
          (\delbar^K \bar{z}^k) \cdot a_k^{(m)}(z)
        \big]
    = K! \cdot a_K^{(m)} (z) ,
  \]
  in contradiction to $a_K^{(m)} \not\equiv 0$.

  Overall, we see that each $a_k$ is holomorphic with $a_k^{(m)} \equiv 0$.
  Thus, $a_k (z) = \sum_{j=0}^{m-1} a_{k,j} \, z^j$ is a polynomial of degree at most $m - 1$,
  and hence $g(z) = \sum_{k,j=0}^{m-1} a_{k,j} \bar{z}^k z^j$
  for suitable coefficients $a_{k,j} \in \CC$.
  Since $\sigma = g$ almost everywhere, this means that $\sigma$ coincides almost everywhere
  with a polynomial in $z$ and $\bar{z}$, contradicting our assumptions.

  \medskip{}

  \noindent
  \textbf{Step 3:} With $\sigma_{a,b}$ as in \Cref{eq:TranslationDilation}, let us define
  \[
    W := \linspan_{\CC}
         \big\{
           \varphi \ast \sigma_{a,b}
           \,\, \colon \,\,
           a,b \in \CC \text{ and } \varphi \in C_c^\infty(\CC; \R)
         \big\}
      \subset C^\infty (\CC; \CC).
  \]
  Furthermore, define $\zeta_m : \CC \to \CC, z \mapsto z^m$ for $m \in \N_0$
  and $R : \CC \to \CC, z \mapsto \Re z$, as well as $\id_{\CC} : \CC \to \CC, z \mapsto z$.
  In this step we show that $W$ is closed under dilations and translations and,
  denoting the closure of $W$ with respect to locally uniform convergence by $\overline{W}$, that
  \begin{equation}
    \id_{\CC}, R \in \overline{W}
    \qquad \text{and} \qquad
    \forall \, m \in \N_0 : \quad
      \zeta_m \in \overline{W} \text{ or } \overline{\zeta_m} \in \overline{W} .
    \label{eq:DeepNetworksWRichness}
  \end{equation}

  To see this, first note that our assumptions on $\sigma$ imply $\sigma \not\equiv 0$;
  choosing $b \in \CC$ with $\sigma(b) \neq 0$, we thus see $\sigma_{0,b} \equiv \sigma(b) \neq 0$
  and hence (with $\eta$ as in Step~1) $W \ni \eta \ast \sigma_{0,b} \equiv \sigma(b) \neq 0$ as well.
  In other words, $W$ contains the constant functions.
  To see that $W$ is closed under translations and dilations, let $\varphi \in C_c^\infty(\CC;\R)$
  and $a,b, \alpha,\beta \in \CC$ be arbitrary.
  In case of $\alpha = 0$, the function $z \mapsto (\varphi \ast \sigma_{a,b})(\beta + \alpha z)$
  is a constant function and thus belongs to $W$, as we just saw.
  Otherwise if $\alpha \neq 0$, then the calculation in \Cref{eq:ConvolutionTransformation}
  shows for $\varphi_\alpha (z) := \varphi (\alpha z)$ that
  \(
    (\varphi \ast \sigma_{a,b}) (\alpha z + \beta)
    = |\alpha|^2 \cdot
      \bigl(
        \varphi_\alpha \ast (\sigma_{a,b})_{\alpha,\beta}
      \bigr) (z) .
  \)
  Since $\varphi_\alpha \in C_c^\infty (\CC; \R)$ and
  $(\sigma_{a,b})_{\alpha,\beta} = \sigma_{a \alpha, b + a \beta}$,
  this implies $\big( z \mapsto  (\varphi \ast \sigma_{a,b}) (\alpha z + \beta)\big) \in W$,
  which easily shows that $W$ is closed under dilations and translations.

  To prove \Cref{eq:DeepNetworksWRichness}, first recall from Step~1 that
  $\del (\varphi_1 \ast \sigma)(\theta_1) \neq 0 \neq \delbar (\varphi_2 \ast \sigma) (\theta_2)$
  for certain $\theta_1, \theta_2 \in \CC$.
  Since $\varphi_1 \ast \sigma, \varphi_2 \ast \sigma \in W$, an easy application of
  \Cref{lem:MonomialExtraction} shows that $\id_{\CC}, \overline{\id_{\CC}} \in \overline{W}$
  and hence $R = \frac{1}{2} (\id_{\CC} + \overline{\id_{\CC}}) \in \overline{W}$.
  Similarly, with $\psi_m$ as in Step~2, we can choose for each $m \in \N_0$ some $\lambda_m \in \CC$
  such that $[\del^m (\psi_m \ast \sigma)](\lambda_m) \neq 0$
  or $\bigl[\delbar^m (\psi_m \ast \sigma)\bigr](\lambda_m) \neq 0$.
  In the first case, \Cref{lem:MonomialExtraction} shows $\zeta_m \in \overline{W}$,
  while in the second case $\overline{\zeta_m} \in \overline{W}$.

  \medskip{}

  \noindent
  \textbf{Step 4:} \emph{(Showing that $\varrho_{\CC} \in \overline{\NN_{\sigma,2}^1}$
  with $\varrho_\CC$ as in \Cref{lem:FromOneDimensionToHigherDimensions}):}
  Let $\eps,r > 0$; we want to construct $\Gamma \in \NN_{\sigma,2}^1$ satisfying
  $\| \varrho_{\CC} - \Gamma \|_{L^\infty (\overline{B_r}(0))} \leq \eps$.

  First of all, note that \Cref{lem:ConvolutionApproximation} shows for arbitrary
  $\varphi \in C_c^\infty(\CC; \R)$ and $a,b \in \CC$ because of $\sigma_{a,b} \in \CalM$ that
  \(
    \varphi \ast \sigma_{a,b}
    \subset \overline{
              \linspan_{\CC}
              \{
                T_z \sigma_{a,b}
                \colon
                z \in \CC
              \}
            }
    \subset \overline{\NN_{\sigma,1}^1} .
  \)
  By definition of $W$, this easily implies $\overline{W} \subset \overline{\NN_{\sigma,1}^1}$.
  Next, since $\varrho : \R \to \R, x \mapsto \max \{ 0, x \}$ is continuous,
  the classical Stone-Weierstraß theorem
  (see \cite[Theorem~7.26]{RudinPrinciplesOfMathematicalAnalysis}) shows that there is a polynomial
  ${p(x) = \sum_{n=0}^N a_n \, x^n}$ (with $N \in \N$ and $a_0,\dots,a_N \in \R$) satisfying
  $|\varrho(x) - p(x)| \leq \frac{\eps}{3}$ for all $x \in [-r, r]$.

  For each $m \in \N_0$ let us choose $\theta_m \in \{ \zeta_m, \overline{\zeta_m} \}$
  with $\theta_m \in \overline{W} \subset \overline{\NN_{\sigma,1}^1}$; this is possible
  by \Cref{eq:DeepNetworksWRichness}.
  Since $\NN_{\sigma,1}^1$ is a vector space, we thus see
  $f := \sum_{n=0}^N a_n \theta_n \in \overline{\NN_{\sigma,1}^1}$;
  hence, there exists $\Phi \in \NN_{\sigma,1}^1$ satisfying
  $\| \Phi - f \|_{L^\infty(\overline{B_{r+1}}(0))} \leq \frac{\eps}{3}$.
  Since $f$ is continuous, it is uniformly continuous on $\overline{B_{r+1}}(0)$;
  we can thus choose $\delta \in (0,1)$ such that $|f(z) - f(w)| \leq \frac{\eps}{3}$
  for all $z,w \in \overline{B_{r+1}}(0)$ with $|z - w| \leq \delta$.
  Note furthermore that $\zeta_m (z) = \overline{\zeta_m (z)} = z^m$ for $z \in \R$,
  and hence $f(z) = p(z)$ for all $z \in \R$.

  Finally, since $R \in \overline{W} \subset \overline{\NN_{\sigma,1}^1}$
  by \Cref{eq:DeepNetworksWRichness}, we can choose $\Psi \in \NN_{\sigma,1}^1$ satisfying
  $|R(z) - \Psi(z)| \leq \delta \leq 1$ for all $z \in \overline{B_r}(0)$.
  Because of $|R(z)| = |\Re z| \leq |z| \leq r$, this implies $\Psi(z) \in \overline{B_{r+1}}(0)$
  for $z \in \overline{B_r}(0)$.
  Summarizing all the obtained estimates, we finally see for $z \in \overline{B_r}(0)$ that
  \[
    \bigl|\varrho_{\CC}(z) - (\Phi \circ \Psi)(z)\bigr|
    \leq \bigl|\varrho(R (z)) - f(R(z))\bigr|
         + \big| f(R(z)) - f(\Psi(z)) \big|
         + \big| f(\Psi(z)) - \Phi(\Psi(z)) \big|
    \leq \eps .
  \]
  Since $\Gamma := \Phi \circ \Psi \in \NN_{\sigma,2}^1$
  by \Cref{lem:NetworkClosureProperties} \ref{enu:NetworksComposition},
  this implies $\varrho_{\CC} \in \overline{\NN_{\sigma,2}^1}$, as claimed.

  \medskip{}

  \noindent
  \textbf{Step 5:} \emph{(Completing the proof):}
  Let $L \in \N_{\geq 2}$ and $d \in \N$.
  With $f^{(a,b)}$ as defined in \Cref{lem:FromOneDimensionToHigherDimensions},
  \Cref{lem:NetworkClosureProperties} shows that
  $\CalG := \NN_{\sigma,L}^d \subset \{ g : \CC^d \to \CC \}$
  is a (complex) vector space that satisfies $f^{(a,b)} \in \NN_{\sigma,L}^d$ for all
  $f \in \CalF := \NN_{\sigma,L}^1$, $b \in \CC$ and $a \in \CC^d$.
  Therefore, \Cref{lem:FromOneDimensionToHigherDimensions} shows that it is enough to prove
  $\varrho_{\CC} \in \overline{\NN_{\sigma,L}^1}$.
  In case of $L = 2$, this was shown in Step~4; thus, let us assume $L > 2$.

  Since \Cref{eq:DeepNetworksWRichness} shows
  $\id_{\CC} \in \overline{W} \subset \overline{\NN_{\sigma,1}^1}$,
  there is a sequence ${(\Phi_n)_{n \in \N} \subset \NN_{\sigma,1}^1}$
  satisfying $\Phi_n \to \id_{\CC}$ locally uniformly on $\CC$.
  Furthermore, by Step~4 we can choose a sequence $(\Psi_n)_{n \in \N} \subset \NN_{\sigma,2}^1$
  satisfying $\Psi_n \to \varrho_{\CC}$ locally uniformly.
  Since the limit functions are continuous, this easily implies
  \(
    \Phi_n \circ \cdots \circ \Phi_n \circ \Psi_n
    \to \id_{\CC} \circ \cdots \circ \id_{\CC} \circ \varrho_{\CC}
    = \varrho_{\CC}
  \)
  with locally uniform convergence, where $\Phi_n$ appears $L-2$ times in the composition.
  A formal argument for the claimed convergence is given in
  \cite[Lemma~A.7]{GribonvalNNApproximationSpaces}.
  Since $\Phi_n \circ \cdots \circ \Phi_n \circ \Psi_n \in \NN_{\sigma,L}^1$ by
  \Cref{lem:NetworkClosureProperties} \ref{enu:NetworksComposition}, we are done.
\end{proof}

\subsection{Necessity for deep networks}%
\label{sub:NecessityDeep}

The following result contains the ``necessary'' part
of \Cref{thm:DeepUniversalApproximationIntroduction};
precisely, it shows that those activation functions $\sigma$
that are excluded in the sufficient criterion of \Cref{thm:DeepNetworksSufficientCriterion}
indeed do \emph{not} yield universal network classes,
at least if one assumes a priori that $\sigma$ is continuous.
As we will see in \Cref{exa:PathologicalUniversalSigma} below,
this continuity assumption can not simply be dropped,
and is thus not just a proof artifact.

\begin{theorem}\label{thm:DeepNetworksNecessaryCriterion}
  Let $\sigma : \CC \to \CC$ be continuous and $d,L \in \N$.
  Assume that (at least) one of the following properties is satisfied:
  \begin{enumerate}[label=\alph*)]
    \item we have $\sigma (z) = p(z, \overline{z})$ for almost all $z \in \CC$,
          where $p \in \CC[X,Y]$ is a complex polynomial of two variables,

    \item we have $\sigma = g$ almost everywhere or $\sigma = \overline{g}$ almost everywhere,
          where $g : \CC \to \CC$ is an entire function.
  \end{enumerate}
  Then there exist $f \in C_c(\CC^d;\R)$ and $\eps > 0$ with $\| f - \Phi \|_{L^1(B_1(0))} \geq \eps$
  for all $\Phi \in \NN_{\sigma,L}^d$.
  In particular, $C(\CC^d; \CC) \nsubseteq \overline{\NN_{\sigma,L}^d}$.
\end{theorem}

\begin{proof}
  Since $\sigma$ is continuous, the identities $\sigma(z) = p(z, \bar{z})$ or $\sigma(z) = g(z)$
  or $\sigma(z) = \overline{g(z)}$ hold everywhere, not only almost everywhere.
  In particular, $\sigma \in C^\infty(\CC; \CC)$, which easily implies
  $\NN_{\sigma,L}^1 \subset C^\infty (\CC; \CC)$ as well.
  We will show below that there is some $m = m(\sigma,L) \in \N$ such that
  \begin{equation}
    \forall \, \Psi \in \NN_{\sigma,L}^1 : \quad \Delta^m \Psi \equiv 0.
    \label{eq:DeepNetworksNecessityPolyharmonicity}
  \end{equation}
  Once this is shown, \Cref{lem:NetworkClosureProperties} \ref{enu:NetworksHighDimToLowDim}
  shows for every $\Phi \in \NN_{\sigma,L}^d$ and arbitrary $w \in \CC^{d-1}$ that
  if we set $e_1 = (1,0,\dots,0) \in \CC^d$ and
  $\Psi : \CC \to \CC, z \mapsto \Phi(z,w) = \Phi \big( (0,w) + z \, e_1 \big)$,
  then $\Psi \in \NN_{\sigma,L}^1$ and hence $\Psi \in C^\infty(\CC;\CC)$
  and $\Delta^m \Psi \equiv 0$.
  Then, a straightforward application of \Cref{lem:PolyharmonicityObstructsUniversality}
  yields the claim of the theorem.

  The (somewhat technical) proof of \Cref{eq:DeepNetworksNecessityPolyharmonicity} is deferred
  to the appendix.
  The intuition is that if $\sigma$ is holomorphic or antiholomorphic,
  then each $\Psi \in \NN_{\sigma,L}^1$ will be holomorphic or antiholomorphic,
  so that $\Delta \Psi \equiv 0$.
  Likewise, if $\sigma(z) = p(z, \overline{z})$ is a polynomial, then each
  $\Psi \in \NN_{\sigma,L}^1$ will be a polynomial in $z,\bar{z}$ of fixed degree $N = N(p,L)$,
  and hence $\Delta^{N+1} \Psi \equiv 0$.
  A more formal argument is given in \Cref{sub:AppendixDeepNetworksNecessityPolyharmonicity}.
\end{proof}

The following example presents an activation function $\sigma \in \CalM$
that coincides almost everywhere with a polynomial, but for which
$\NN_{\sigma,L}^d$ is nevertheless universal for any $L \geq 2$.

\begin{example}\label{exa:PathologicalUniversalSigma}
  Let $\varrho_{\CC} : \CC \to \CC, z \mapsto \max \{ 0, \Re z \}$
  as in \Cref{lem:FromOneDimensionToHigherDimensions}.
  Define
  \[
    \sigma : \quad
    \CC \to \CC, \quad
    z \mapsto \begin{cases}
                \Re z,            & \text{if } z \notin \R , \\
                \varrho_{\CC}(z), & \text{if } z \in \R .
              \end{cases}
  \]
  Clearly, $\sigma$ is locally bounded.
  Moreover, we have $N := \{ z \in \CC \colon \sigma \text{ not continuous at } z \} \subset \R$,
  so that the closure $\overline{N} \subset \R$ is a null-set when considered as a subset of $\CC$;
  therefore, $\sigma \in \CalM$.
  Next, note that $\sigma(z) = \Re z = z + \bar{z}$ for almost all $z \in \CC$
  (namely, for all $z \in \CC \setminus \R$),
  so that $\sigma$ satisfies the first condition of \Cref{thm:DeepNetworksNecessaryCriterion}.

  Nevertheless, it is straightforward to verify that
  $\sigma \circ \cdots \circ \sigma = \varrho_{\CC}$,
  whenever $\sigma$ appears at least twice in the composition.
  In view of \Cref{lem:NetworkClosureProperties} \ref{enu:NetworksComposition},
  we thus see $\varrho_{\CC} \in \NN_{\sigma,L}^1$ whenever $L \in \N_{\geq 2}$.
  Via a straightforward application of \Cref{lem:FromOneDimensionToHigherDimensions},
  this implies $C(\CC^d; \CC) \subset \overline{\NN_{\sigma,L}^d}$
  for all $d \in \N$ and $L \in \N_{\geq 2}$.

  This example shows that in the necessary condition for the universality of deeper networks
  (\Cref{thm:DeepNetworksNecessaryCriterion}), one cannot simply relax the continuity assumption
  by instead assuming that $\sigma \in \CalM$.
\end{example}

\appendix

\section{Postponed proofs}%
\label{sec:Appendix}

\subsection{Proof of Proposition~\ref{prop:IntroHolomorphicIsBad}}%
\label{sub:ProofOfHolomorphicIsBad}

\begin{proof}
  Let $f \in C_c(\CC; \CC)$ be arbitrary such that $f|_{B_1(0)}$ is not smooth;
  for instance, take ${f(z) = \max \bigl\{ 0, 1 - |z| \bigr\}}$.
  Assume towards a contradiction that the claim of the proposition fails;
  there is thus a sequence $(\Theta_n)_{n \in \N}$ of admissible weights such that
  if we set $D_n := D_{\Theta_n}$ and $\Phi_n := \Phi_{\Theta_n}$, then
  $\sup_{z \in B_1(0) \setminus D_{n}} |f(z) - \Phi_n(z)| \leq n^{-1} < \infty$
  for all $n \in \N$.

  Note that $D_n \subset \CC$ is closed and discrete, so that $B_1(0) \cap D_n$ is finite.
  Furthermore, ${\Phi_n : \CC \setminus D_n \to \CC}$ is holomorphic, meaning that
  each $z \in D_n$ is an isolated singularity of $\Phi_n$.
  Because of ${|\Phi_n(z)| \leq 1 + \| f \|_{L^\infty}}$ for all $z \in B_1(0) \setminus D_n$,
  we see that $\Phi_n$ is bounded on a neighborhood of each $z \in D_n \cap B_1(0)$.
  By elementary complex analysis (see for instance \cite[Chapter~3, Theorem~3.1]{SteinComplexAnalysis})
  this implies that $\Phi_n|_{B_1(0)}$ can be extended to a holomorphic function
  $\Psi_n : B_1(0) \to \CC$.
  By continuity (and since $B_1(0) \setminus D_n$ is dense in $B_1(0)$), we thus see
  \[
    \sup_{z \in B_1(0)}
      |f(z) - \Psi_n(z)|
    = \sup_{z \in B_1(0) \setminus D_n}
        |f(z) - \Psi_n(z)|
    = \sup_{z \in B_1(0) \setminus D_n}
        |f(z) - \Phi_n(z)|
    \leq n^{-1} ,
  \]
  so that $\Psi_n$ converges to $f$ uniformly on $B_1(0)$.
  Again by elementary complex analysis
  (see \cite[Chapter~2, Theorem~5.2 and Corollary~4.2]{SteinComplexAnalysis}),
  this implies that $f|_{B_1(0)}$ is holomorphic and hence smooth,
  contradicting our choice of $f$.
\end{proof}

\subsection{A technical result used to prove Lemma~\ref{lem:MonomialExtraction}}%
\label{sub:DerivativesViaDividedDifferences}

\begin{proposition}\label{prop:DerivativesInTheSpace}
  Let $V \subset C^\infty (\CC; \CC)$ be a (complex) vector space that is closed
  under dilations and translations.
  Then for each $\varphi \in V$, $\theta \in \CC$, and $n,m \in \N_0$, we have
  \[
    \Big(
      z \mapsto \frac{\partial^{n+m}}{\partial a^n \, \partial b^m} \Big|_{a = b = 0} \,
                  \varphi\bigl( (a + i b) \, z + \theta\bigr)
    \Big)
    \in \overline{V},
  \]
  with $\overline{V}$ as defined in \Cref{eq:ClosureDefinition}.
\end{proposition}

\begin{proof}
  \textbf{Step~1:} Let $\gamma \in C^\infty(\CC^2; \CC)$ and define
  $W_\gamma := \linspan_{\CC} \{ \gamma(w, \bullet) \colon w \in \CC \} \subset \{ f : \CC \to \CC \}$.
  Given $z \in \CC$, define $\gamma_z \in C^\infty (\CC; \CC)$ by $\gamma_z(w) = \gamma(w,z)$.
  In this step we show that
  \begin{equation}
    \big(
      z \mapsto (\partial_j \gamma_z) (w_0)
    \big)
    \in \overline{W_\gamma}
    \qquad \forall \, j \in \{ 1,2 \} \text{ and } w_0 \in \CC,
    \label{eq:DerivativeInSpaceAuxiliary1}
  \end{equation}
  where the closure $\overline{W_\gamma}$ is as defined in \Cref{eq:ClosureDefinition},
  and where, given a function $f : \CC \to \CC$,
  we write $\partial_1 f$ and $\partial_2 f$ for the (real) partial derivatives of $f(z)$
  with respect to the real- and imaginary parts of $z$, respectively.

  To prove \Cref{eq:DerivativeInSpaceAuxiliary1},
  let $R, \eps > 0$, $w_0 \in \CC$, and $j \in \{1,2\}$.
  Set
  \[
    \underline{\gamma} :
    \R^4 \to \CC,
    (a,b,x,y) \mapsto \gamma(a + i b, x + i y) ,
    \,\,\, \text{so that} \,\,\,
    \partial_j \gamma_z (w)
    = (\partial_j \underline{\gamma}) (\Re w, \Im w, \Re z, \Im z) .
  \]
  Using this identity, and since $\partial_j \underline{\gamma}$ is uniformly continuous
  on the compact set $\overline{B_1}(w_0) \times \overline{B_R}(0)$,
  we can choose $h \in (0, 1)$ such that
  $|\partial_j \gamma_z (w_0) - \partial_j \gamma_z (w_0 + t \, v)| \leq \eps$
  for all $z \in \overline{B_R}(0)$, all $v \in \CC$ with $|v| \leq 1$, and all $0 \leq t \leq h$.
  Now, define $v_1 := 1 \in \CC$ and $v_2 := i \in \CC$.
  Then we see for every $z \in \overline{B_R}(0)$ that
  \begin{align*}
    \Big|
      (\partial_j \gamma_z) (w_0)
      - \frac{\gamma (w_0 + h \, v_j, z) - \gamma (w_0, z)}{h}
    \Big|
    & = \Big|
          (\partial_j \gamma_z) (w_0)
          - \frac{\gamma_z (w_0 + h \, v_j) - \gamma_z (w_0)}{h}
        \Big| \\
    & = \Big|
          \frac{1}{h}
          \int_0^h
            (\partial_j \gamma_z) (w_0)
            - (\partial_j \gamma_z) (w_0 + t \, v_j)
          \, d t
        \Big|
      \leq \eps .
  \end{align*}
  Since $\eps > 0$ was arbitrary, and since
  $h^{-1} \cdot \bigl(\gamma(w_0 + h \, v_j , \bullet) - \gamma(w_0, \bullet)\bigr) \in W_\gamma$,
  this proves \Cref{eq:DerivativeInSpaceAuxiliary1}.

  \medskip{}

  \noindent
  \textbf{Step~2:} With notation as in Step~1, we show for $\gamma \in C^\infty(\CC^2; \CC)$ that
  \begin{equation}
    \big(
      z \mapsto (\partial^\alpha \gamma_z) (w_0)
    \big)
    \in \overline{W_\gamma}
    \qquad \forall \, w_0 \in \CC \text{ and } \alpha \in \N_0^2 .
    \label{eq:DerivativeInSpaceAuxiliary2}
  \end{equation}
  We prove this by induction on the multiindex $\alpha \in \N_0^2$,
  where the case $\alpha = 0$ is trivial.
  Furthermore, if \Cref{eq:DerivativeInSpaceAuxiliary2} holds for some $\alpha \in \N_0^2$
  and if $j \in \{ 1,2 \}$, then applying Step~1 with
  $\widetilde{\gamma}(w,z) = (\partial^\alpha \gamma_z)(w)$ instead of $\gamma$ shows that
  \[
    \big(
      z \mapsto (\partial^{\alpha + e_j} \gamma_z) (w_0)
    \big)
    = \big(
        z \mapsto (\partial_j \widetilde{\gamma}_z) (w_0)
      \big)
    \in \overline{W_{\widetilde{\gamma}}}
    = \overline{\strut
        \linspan_{\CC} \big\{ \widetilde{\gamma}(w, \bullet) \colon w \in \CC \big\}
      } .
  \]
  But by induction we have
  \(
    \widetilde{\gamma}(w,\bullet)
    = \big(
        z \mapsto (\partial^\alpha \gamma_z) (w)
      \big)
    \in \overline{W_\gamma}
    \vphantom{\sum_j}
  \)
  for all $w \in \CC$.
  In combination, we thus see that
  \(
    \strut
    \big(
      z \mapsto (\partial^{\alpha + e_j} \gamma_z) (w_0)
    \big)
    \in \overline{W_\gamma}
    ,
  \)
  completing the induction.

  \medskip{}

  \noindent
  \textbf{Step~3:} In this step, we finish the proof.
  Let $V, \varphi, \theta, n,m$ as in the claim of the proposition.
  Define $\gamma : \CC^2 \to \CC, (w,z) \mapsto \varphi(w \, z + \theta)$.
  Clearly, $\gamma \in C^\infty(\CC^2; \CC)$, since $\varphi \in C^\infty$.
  Therefore, Step~2 shows with $\varphi_{w,\theta}$ as in \Cref{eq:TranslationDilation} that
  \begin{align*}
    \bigg(
      z \mapsto \frac{\partial^{n+m}}{\partial a^n \, \partial b^m} \Big|_{a = b = 0}
                  \varphi \big( (a + i b) z + \theta \big)
    \bigg)
    & = \big(
          z \mapsto (\partial^{(n,m)} \gamma_z) (0)
        \big) \\
    & \in \overline{
            \linspan_{\CC} \big\{ \gamma(w, \bullet) \colon w \in \CC \big\}
          } \\[0.2cm]
    & =   \overline{
            \linspan_{\CC} \big\{ \varphi_{w,\theta} \colon w \in \CC \big\}
          }
      \subset \overline{V} ,
  \end{align*}
  since $\varphi \in V$ and since $V$ is a $\CC$-vector space that is
  closed under dilations and translations.
\end{proof}

\subsection{Inverse images of null-sets under linear forms remain null-sets}%
\label{sub:NullSetLinearInverseImage}

\begin{lemma}\label{lem:HyperplaneNullSet}
  Let $N \subset \CC \cong \R^2$ be a (Lebesgue) null-set
  and let $d \in \N$, $a \in \CC^d \setminus \{ 0 \}$, and $b \in \CC$.
  Then
  \[
    M := \big\{
           z \in \CC^d
           \quad\colon\quad
           a^T z + b \in N
         \big\}
    \subset \CC^d \cong \R^{2 d}
  \]
  is a null-set as well.
\end{lemma}

\begin{proof}
  Since each null-set is contained in a Borel-measurable null-set, we can assume that
  $N$ is Borel-measurable.
  This implies that $M$ is Borel-measurable as well.

  For each $j \in \FirstN{d}$, write $a_j = A_j + i B_j$ and $b = \alpha + i \beta$,
  as well as ${z_j = x_j + i \, y_j}$ with $A_j, B_j, \alpha, \beta, x_j, y_j \in \R$.
  Choose $\ell \in \FirstN{d}$ with $a_\ell \neq 0$.
  Given $x \in \R^{d-1}$ and $x' \in \R$, let us define
  \[
    \nu (x, x')
    := \bigl( x_1,\dots, x_{\ell-1}, x', x_\ell, \dots, x_{d-1} \bigr)
    \in \R^d .
  \]
  Then, Tonelli's theorem shows for the $2d$-dimensional Lebesgue measure $\lambda_{2 d}$ that
  \begin{equation}
    \lambda_{2d}(M)
    = \int_{\R^{d-1} \times \R^{d-1}}
        \int_{\R^2}
          \Indicator_M \big( \nu (x, x') + i \, \nu (y, y') \big)
        \, d (x', y')
      \, d (x, y) .
    \label{eq:HyperplaceNullSetTonelliArgument}
  \end{equation}

  Now, given fixed $x,y \in \R^{d-1}$ and identifying a complex number $w = a + i \, b$
  with the vector $(a,b) \in \R^2$, we see because of
  \(
    (A_\ell + i \, B_\ell) \cdot (x' + i \, y')
    = A_\ell \, x' - B_\ell \, y' + i \cdot (A_\ell \, y' + B_\ell \, x')
  \)
  that there exists a vector $\theta = \theta(x,y,a,b) \in \R^2$ such that
  \[
    b + a^T \big[ \nu(x,x') + i \, \nu(y,y') \big]
    = b + \theta
      + \begin{pmatrix}
          A_\ell & - B_\ell \\
          B_\ell & A_\ell
        \end{pmatrix}
        \begin{pmatrix}
          x' \\ y'
        \end{pmatrix}
    =: b + \theta + \Gamma \begin{pmatrix} x' \\ y' \end{pmatrix} ,
  \]
  where one easily sees $\det \Gamma = A_\ell^2 + B_\ell^2 = |a_\ell|^2 > 0$.
  The above identity shows that
  \begin{align*}
    \nu(x, x') + i \, \nu(y, y') \in M \quad
    & \Longleftrightarrow \quad
    b + a^T \big[ \nu(x,x') + i \, \nu(y,y') \big] \in N \\
    & \Longleftrightarrow \quad
    (x', y') \in \Gamma^{-1} \cdot (N - b - \theta) ,
  \end{align*}
  where elementary properties of the Lebesgue measure show that
  $\Gamma^{-1} \cdot (N - b - \theta)$ is a null-set, since $N$ is.
  In other words, we have shown with the $2$-dimensional Lebesgue measure $\lambda_2$ that
  \[
    \int_{\R^2}
      \Indicator_M \big( \nu (x, x') + i \, \nu (y, y') \big)
    \, d (x', y')
    = \lambda_2 \bigl( \Gamma^{-1} \cdot (N - b - \theta) \bigr)
    = 0
    \quad \forall \, (x,y) \in \R^{d-1} \times \R^{d-1} .
  \]
  Combined with \Cref{eq:HyperplaceNullSetTonelliArgument},
  this shows as claimed that $\lambda_{2d}(M) = 0$.
\end{proof}

\subsection{Weakly polyharmonic functions are smooth}%
\label{sub:WeaklyPolyharmonicIsPolyharmonic}

In this section, we show that each ``weakly polyharmonic function''
coincides almost everywhere with a smooth polyharmonic function.

\begin{lemma}\label{lem:WeaklyPolyharmonicIsPolyharmonic}
  Let $\emptyset \neq U \subset \CC$ be open, let $m \in \N_0$,
  and let $f \in L_{\loc}^1(U; \R)$ be weakly polyharmonic of order $m$, in the sense that
  \begin{equation}
    \int_{U}
      f(z) \cdot \Delta^m \varphi (z)
    \, d z
    = 0
    \qquad \forall \, \varphi \in C_c^\infty (U; \R) .
    \label{eq:WeakPolyharmonicity}
  \end{equation}
  Then there exists a function $g \in C^\infty (U; \R)$ with $\Delta^m g \equiv 0$
  satisfying $f = g$ almost everywhere.
\end{lemma}

\begin{rem*}
  Applying the lemma separately to $\Re f$ and $\Im f$, one sees that if $f \in L_{\loc}^1(U; \CC)$
  and $\int_{U} f(z) \Delta^m \varphi(z) \, d z = 0$ for all $\varphi \in C_c^\infty (U; \R)$,
  then there exists $g \in C^\infty(U; \CC)$ with $\Delta^m g \equiv 0$ satisfying
  $f = g$ almost everywhere.
\end{rem*}

\begin{proof}
  We first prove by induction on $m \in \N_0$ that if $\psi \in \CalD'(U)$ is a distribution
  satisfying $\Delta^m \psi = 0$, then there exists a smooth function $g \in C^\infty(U; \R)$
  such that $\psi = g$ as distributions; here, we only consider ``real-valued'' distributions,
  meaning that $\psi : C_c^\infty(U; \R) \to \R$ is a linear functional that is continuous with
  respect to the usual topology on $C_c^\infty(U; \R)$.
  We refer to \cite[Chapter~6]{RudinFA} regarding the details of this topology
  and the formalism of distributions.

  For the case $m = 0$, the assumption $\Delta^m \psi = 0$ simply means $\psi = 0$,
  so that the claim is trivial.
  Now, assume the claim holds for some $m \in \N_0$, and let $\psi \in \CalD'(U)$ with
  $\Delta^{m+1} \psi = 0$.
  Define $\widetilde{\psi} := \Delta \psi \in \CalD'(U)$ and note $\Delta^m \widetilde{\psi} = 0$.
  Thus, by the induction hypothesis, there is $h \in C^\infty (U; \R)$
  with $\Delta \psi = \widetilde{\psi} = h$.
  Now, \emph{Weyl's lemma} (see for instance \cite[§2]{StroockWeylLemma}), implies that
  $\psi = g$ (in the sense of distributions) for some $g \in C^\infty(U; \R)$.

  Finally, we prove the lemma.
  The assumption in \Cref{eq:WeakPolyharmonicity} means that ${\Delta^m f = 0}$
  in the sense of distributions.
  By what we showed above, this implies existence of a function $g \in C^\infty(U; \R)$
  with $f = g$ as distributions, i.e.,
  ${\int_{U} f(z) \, \varphi(z) \, dz = \int_{U} g(z) \, \varphi(z) \, d z}$
  for all $\varphi \in C_c^\infty (U; \R)$.
  By the \emph{fundamental lemma of calculus of variations}
  (see for instance \cite[Lemma~4.22]{AltFA}), this implies $f = g$ almost everywhere.
  Finally, we see by partial integration that
  \(
    \int_{U} \Delta^m g (z) \, \varphi(z) \, d z
    = \int_{U} g (z) \, \Delta^m \varphi(z) \, d z
    = \int_{U} f (z) \, \Delta^m \varphi(z) \, d z
    = 0
  \)
  for all $\varphi \in C_c^\infty (U; \R)$.
  By the fundamental lemma of calculus of variations, this implies $\Delta^m g = 0$
  almost everywhere, and hence everywhere by continuity.
\end{proof}

\subsection{Proof of Lemma~\ref{lem:NetworkClosureProperties}}%
\label{sub:NetworkClosurePropertiesProof}

\begin{proof}[Proof of \Cref{lem:NetworkClosureProperties}]
  The whole proof will use the convention that $\Phi = \CalN_\sigma \Theta$ and
  $\Psi = \CalN_\sigma \Lambda$ with $\Theta = \bigl( (A_0,b_0),\dots,(A_L,b_L) \bigr)$
  and $\Lambda = \big( (B_0,c_0),\dots,(B_T, c_T) \big)$, where possibly $T = L$.
  Furthermore, $N_\ell, M_t \in \N$ are chosen such that $A_\ell \in \CC^{N_{\ell+1} \times N_\ell}$
  and $b_\ell \in \CC^{N_{\ell+1}}$ as well as $B_t \in \CC^{M_{t+1} \times M_t}$
  and $c_t \in \CC^{M_{t+1}}$ for $\ell \in \{ 0,\dots,L \}$ and $t \in \{ 0,\dots,T \}$.

  \medskip{}

  \noindent
  \textbf{Ad \ref{enu:NetworksVectorSpace}}
  In this case, $T = L$ and furthermore $N_0 = M_0 = d$ and $N_{L+1} = M_{L+1} = 1$.
  Let us define $K_0 := d$, $K_{L+1} := 1$ and $K_\ell := N_\ell + M_\ell$
  for $\ell \in \{ 1,\dots,L \}$.
  In addition, define
  \[
    C_0 := \left( \begin{matrix} A_0 \\ B_0 \end{matrix} \right) \in \CC^{K_1 \times K_0}
    \quad \text{and} \quad
    C_\ell := \left(
                \begin{matrix}
                  A_\ell & 0 \\
                  0      & B_\ell
                \end{matrix}
              \right)
           \in \CC^{K_{\ell + 1} \times K_\ell}
    \quad\text{for}\quad
    \ell \in \{ 1,\dots,L-1 \} ,
  \]
  as well as $C_L := \bigl(\alpha A_L \,\big|\, \beta B_L\bigr) \in \CC^{1 \times K_L}$,
  and furthermore
  \[
    e_\ell := \left( \begin{matrix} b_\ell \\ c_\ell \end{matrix} \right) \in \CC^{K_{\ell+1}}
    \quad \text{for} \quad \ell \in \{ 0,\dots,L-1 \}
    \qquad \text{and} \qquad
    e_L := \alpha \, b_L + \beta \, c_L \in \CC^{K_{L+1}} .
  \]
  Setting $\Gamma := \big( (C_0,e_0),\dots,(C_L,e_L) \big)$,
  it is straightforward but slightly tedious to verify that
  ${\CalN_\sigma \Gamma \in \NN_{\sigma,L}^d}$ and
  \(
    \CalN_\sigma \Gamma
    = \alpha \, \CalN_\sigma \Theta + \beta \, \CalN_\sigma \Lambda
    = \alpha \, \Phi + \beta \, \Psi
    .
  \)

  \medskip{}

  \noindent
  \textbf{Ad \ref{enu:NetworksLowDimToHighDim}}
  In this case, $N_0 = 1$, so that ${\widetilde{A_0} := A_0 \, a^T \in \CC^{N_1 \times d}}$ and
  ${\widetilde{b_0} := b_0 + A_0 \, b \in \CC^{N_1}}$ are well-defined and satisfy
  $\widetilde{b_0} + \widetilde{A_0} z = b_0 + A_0 \bigl(b + a^T z\bigr)$ for $z \in \CC^d$.
  Based on this, it is easy to verify for
  $\Gamma := \big( (\widetilde{A_0}, \widetilde{b_0}), (A_1,b_1),\dots,(A_L,b_L) \big)$
  that ${\CalN_\sigma \Gamma (z) = (\CalN_\sigma \Theta)(b + a^T z) = \Phi(b + a^T z) = \Xi(z)}$,
  and hence $\Xi \in \NN_{\sigma,L}^d$.

  \medskip{}

  \noindent
  \textbf{Ad \ref{enu:NetworksHighDimToLowDim}}
  In this case, $N_0 = d$, so that $\widetilde{A_0} := A_0 \, a \in \CC^{N_1 \times 1}$
  and $\widetilde{b_0} := b_0 + A_0 \, b \in \CC^{N_1}$ are well-defined and satisfy
  $b_0 + A_0(b + z a) = \widetilde{b_0} + \widetilde{A_0} z$ for all $z \in \CC$.
  Using this identity, it is easy to verify for
  $\Gamma := \big( (\widetilde{A_0}, \widetilde{b_0}), (A_1,b_1),\dots,(A_L,b_L) \big)$
  that $\CalN_\sigma \Gamma (z) = (\CalN_\sigma \Theta) (b + z a) = \Xi(z)$ for all $z \in \CC$,
  and hence $\Xi \in \NN_{\sigma,L}^1$.

  \medskip{}

  \noindent
  \textbf{Ad \ref{enu:NetworksComposition}}
  Note that $N_0 = d$ and $N_{L+1} = 1 = M_0$, so that $C_L := B_0 A_L \in \CC^{M_1 \times N_L}$
  and $e_L := c_0 + B_0 \, b_L \in \CC^{M_1}$ are well-defined.
  Furthermore, let us define $C_\ell := A_\ell \text{ and } e_\ell := b_\ell$
  for $\ell \in \{ 0,\dots,L-1 \}$, as well as
  $C_{t + L} := B_t \text{ and } e_{t + L} := c_t$ for $t \in \{ 1,\dots,T \}$.
  It is straightforward to verify for $\Gamma = \big( (C_0,e_0),\dots,(C_{L+T}, e_{L+T}) \big)$
  that $\CalN_\sigma \Gamma = (\CalN_\sigma \Lambda) \circ (\CalN_\sigma \Theta) = \Psi \circ \Phi$
  and hence $\Psi \circ \Phi \in \CalN_{\sigma, L + T}^d$, as claimed.
\end{proof}

\subsection{Technical details of the proof of Theorem~\ref{thm:DeepNetworksNecessaryCriterion}}%
\label{sub:AppendixDeepNetworksNecessityPolyharmonicity}

\begin{proof}[Proof of \Cref{eq:DeepNetworksNecessityPolyharmonicity}
  in the proof of \Cref{thm:DeepNetworksNecessaryCriterion}]

  Let $\Psi \in \NN_{\sigma,L}^1$ with $\Psi = \CalN_\sigma \Theta$ where
  $\Theta = \big( (A_0,b_0),\dots,(A_L,b_L) \big)$ with $A_\ell \in \CC^{N_{\ell+1} \times N_\ell}$
  and $b_\ell \in \CC^{N_{\ell+1}}$, where ${N_0 = N_{L+1} = 1}$.
  We define $u_{-1} (z) := z$ and inductively
  \[
    v_\ell (z) := b_\ell + A_\ell \, u_{\ell-1}(z) \in \CC^{N_{\ell+1}}
    \qquad \text{and} \qquad
    u_\ell (z) := \sigma(v_\ell(z)) \in \CC^{N_{\ell+1}}
    \quad \text{for } \ell \in \{ 0,\dots,L \} ,
  \]
  with $\sigma$ applied componentwise.
  Then $\Psi(z) = v_L (z)$.

  \medskip{}

  \noindent
  \textbf{Proof of \Cref{eq:DeepNetworksNecessityPolyharmonicity} in Case a)}
  In this case there exists $p \in \CC[X,Y]$ satisfying ${\sigma(z) = p(z, \bar{z})}$
  for all $z \in \CC$.
  With $N \geq \deg p$, we will prove \Cref{eq:DeepNetworksNecessityPolyharmonicity}
  for ${m = N^{L}+1}$.
  In fact, we will show by induction on $\ell \in \{ 0,\dots,L \}$ that for each
  $k \in \{ 1,\dots,N_{\ell+1} \}$ there exists a polynomial $q_{\ell,k} \in \CC[X,Y]$
  with $\deg q_{\ell,k} \leq N^\ell$ and $(v_\ell (z))_k = q_{\ell,k}(z,\bar{z})$.
  Once this is shown, it follows that
  \(
    \Psi(z)
    = (v_L(z))_1
    = q_{L,1}(z,\bar{z})
    = \sum_{k,\ell=0}^{N^L}
        b_{k,\ell} \, z^k \bar{z}^\ell
  \)
  for certain $b_{k,\ell} \in \CC$.
  In view of \Cref{eq:WirtingerLaplaceRepresentation} and because of
  $\del^m \delbar^m (z^k \bar{z}^\ell) = (\del^m z^k) \cdot \overline{(\del^m z^\ell)} = 0$
  for $m > \max \{ k,\ell \}$, this will prove \Cref{eq:DeepNetworksNecessityPolyharmonicity}
  for $m = N^L + 1$.

  To prove the existence of the $q_{\ell,k}$, given a polynomial
  ${q(X,Y) = \sum_{k,\ell} a_{k,\ell} \, X^k Y^\ell}$, let us define
  ${\widetilde{q}(X,Y) := \sum_{k,\ell} \overline{a_{k,\ell}} \, X^\ell Y^k}$.
  It is then straightforward to verify that $\overline{q(z, \bar{z})} = \widetilde{q}(z, \bar{z})$.
  For $\ell = 0$, note that $(v_0(z))_k = q_{0,k}(z,\bar{z})$ for
  $q_{0,k}(X,Y) := (b_0)_k + (A_0)_{k,1} \, X$.
  Next, assume that the claim holds for some $\ell \in \{ 0,\dots,L-1 \}$.
  For $k \in \{ 1,\dots,N_{\ell+1} \}$, define
  $r_{\ell,k} (X,Y) := p\bigl(q_{\ell,k}(X,Y), \widetilde{q_{\ell,k}} (X,Y)\bigr)$.
  It is easy to see $r_{\ell,k} \in \CC[X,Y]$ and furthermore
  \({
    \deg r_{\ell,k}
    \leq \deg (p)
         \cdot \max
               \{
                 \deg q_{\ell,k},
                 \deg \widetilde{q_{\ell,k}}
               \}
    \leq N^{\ell+1} .
  }\)
  Moreover,
  \begin{align*}
    \bigl(u_{\ell} (z)\bigr)_k
    & = \sigma\bigl( (v_\ell(z))_k \bigr)
      = \sigma\bigl(q_{\ell,k}(z,\bar{z})\bigr)
      = p\bigl(\, q_{\ell,k}(z,\bar{z}), \,\, \overline{q_{\ell,k}(z,\bar{z})} \,\bigr) \\
    & = p\bigl(\, q_{\ell,k}(z,\bar{z}), \,\, \widetilde{q_{\ell,k}}(z,\bar{z}) \,\bigr)
      = r_{\ell,k}(z, \bar{z}) .
  \end{align*}
  Finally, it is straightforward to check for
  \(
    q_{\ell+1,k}(X,Y)
    := (b_{\ell+1})_k
       + \sum_{j=1}^{N_{\ell+1}}
           (A_{\ell+1})_{k,j} \, r_{\ell,j}(X, Y)
  \)
  that $\deg q_{\ell+1,k} \leq N^{\ell+1}$ and $(v_{\ell+1}(z))_k = q_{\ell+1,k}(z,\bar{z})$,
  as claimed.

  \medskip{}

  \noindent
  \textbf{Proof of \Cref{eq:DeepNetworksNecessityPolyharmonicity} in Case b)}
  In this case, we prove \Cref{eq:DeepNetworksNecessityPolyharmonicity} for $m = 1$.
  Indeed, we will prove by induction on $\ell \in \{ 0,\dots,L \}$
  that all components $(v_\ell)_k : \CC \to \CC$ of $v_\ell$ are jointly holomorphic
  or jointly antiholomorphic.
  Here we say that $f : \CC \to \CC$ is antiholomorphic if $\overline{f}$ is holomorphic.

  For the case $\ell = 0$, note that all components $(v_0 (z))_k = (b_0)_k + (A_0)_{k,1} z$
  are holomorphic.
  For the induction step, let $\ell \in \{ 0,\dots, L-1 \}$ and assume that
  all $(v_\ell)_k$ are holomorphic or all $(v_\ell)_k$ are antiholomorphic.
  Note that if $f : \CC \to \CC$ is holomorphic, then so is $z \mapsto \overline{f(\bar{z})}$;
  this can be seen for instance using the power series expansion
  $\overline{f(\bar{z})} = \sum_{n=0}^\infty \overline{a_n} z^n$
  for ${f(z) = \sum_{n=0}^\infty a_n z^n}$.
  Thus, for functions $f,g : \CC \to \CC$ that are each holomorphic or antiholomorphic,
  there are four cases:
  \emph{i)} if $f,g$ are holomorphic, then so is $f \circ g$;
  \emph{ii)} if $f$ is holomorphic and $g$ antiholomorphic, then
  $\overline{f(g(z))} = \overline{f(\overline{\bar{g}(z)})}$
  is holomorphic and hence $f \circ g$ is antiholomorphic;
  \emph{iii)} if $f$ is antiholomorphic and $g$ holomorphic, then
  $\overline{f \circ g} = \overline{f} \circ g$ is holomorphic,
  so that $f \circ g$ is antiholomorphic;
  \emph{iv)} if $f,g$ are both antiholomorphic,
  then $f(g(z)) = \overline{\bar{f}(\overline{\bar{g}(z)})}$ is holomorphic.
  Overall, these considerations show that $(u_\ell (z))_k = \sigma\bigl( (v_\ell(z))_k \bigr)$
  is holomorphic for all $k$ or antiholomorphic for all $k$.
  Since
  \(
    (v_{\ell+1}(z) )_k
    = (b_{\ell+1})_k
      + \sum_{j =1}^{N_{\ell+1}}
          (A_{\ell+1})_{k,j} \, (u_\ell (z))_j
    ,
  \)
  this shows that all component functions of $v_{\ell+1}$ are jointly holomorphic
  or jointly antiholomorphic.

  To complete the proof, note that $\Psi = (v_L)_1$ is holomorphic or antiholomorphic.
  As seen in \Cref{sec:WirtingerCalculus}, we have $\delbar f \equiv 0$
  if $f$ is holomorphic, and hence $\del f = \overline{\delbar \overline{f}} \equiv 0$
  if $f$ is antiholomorphic.
  In any case, \Cref{eq:WirtingerLaplaceRepresentation} shows $\Delta f \equiv 0$ if
  $f$ is holomorphic or antiholomorphic; in particular $\Delta \Psi \equiv 0$, as claimed.
\end{proof}

\section*{Acknowledgments}

The discussions with Götz Pfander, Dae Gwan Lee and Andrei Caragea that led to the questions
answered in this paper are greatly appreciated.
Many thanks are due to Andrei Caragea for his help in improving the presentation of this paper.
The author acknowledges support by the German Science Foundation (DFG)
in the context of the Emmy Noether junior research group VO 2594/1-1.

\bibliographystyle{abbrvurl}

{\footnotesize
\bibliography{references}
}

\end{document}